\documentclass[leqn,10pt]{amsart}
\usepackage[utf8]{inputenc}
\usepackage{a4}
\usepackage{amsmath,amssymb,amscd,bbm}
\usepackage{enumerate}
\usepackage[mathscr]{eucal}
\usepackage[all]{xy}

\newtheorem{thm}{Theorem}[section]

\newtheorem{prop}[thm]{Proposition}
\newtheorem{cor}[thm]{Corollary}

\theoremstyle{definition}

\newtheorem{exas}[thm]{Examples}

\theoremstyle{remark}

\numberwithin{equation}{section}

\newcommand{\vanish}[1]{\relax}       
\def\qedsymbol{\hbox to 1ex{\llap{\rule{0.25pt}{1ex}}\rlap{\rule{1ex}{0.25pt}}\lower0.25pt\rlap{\raise1ex\rlap{\rule{1ex}{0.25pt}}}\hskip1ex\llap{\rule{0.25pt}{1ex}}}}
\def\rlqed{\rlap{\rule{\hsize}{0pt}\kern-1ex\kern-1em\qed}} 

\newcounter{aufzi}
\newenvironment{aufzi}{\begin{list}{ {\upshape\alph{aufzi})}}{
        \usecounter{aufzi}
        \topsep1ex
        \parsep0cm
        \itemsep0.8ex
        \leftmargin1cm
        \labelwidth0.5cm
        \labelsep0.3cm
}}
{\end{list}}

\newcounter{aufzii}

\newcounter{aufziii}
\newenvironment{aufziii}{\begin{list}{ {\upshape\arabic{aufziii})}}{
        \usecounter{aufziii}
        \topsep1ex
        \parsep0cm
        \itemsep0.8ex
        \leftmargin1cm
        \labelwidth0.5cm
        \labelsep0.3cm
}}
{\end{list}}



\newcommand{\calA}{\mathcal{A}}
\newcommand{\calB}{\mathcal{B}}

\newcommand{\calF}{\mathcal{F}}
\newcommand{\calG}{\mathcal{G}}
\newcommand{\calH}{\mathcal{H}}

\newcommand{\calM}{\mathcal{M}}

\newcommand{\calU}{\mathcal{U}}
\newcommand{\calV}{\mathcal{V}}

\newcommand{\scrP}{\mathscr{P}}

\newcommand{\frakM}{\mathfrak{M}}

\def\bfA{\mathbf{A}}
\def\bfB{\mathbf{B}}
\def\bfC{\mathbf{C}}

\def\bfV{\mathbf{V}}							
\def\bfW{\mathbf{W}}	
\def\bfX{\mathbf{X}}		
\def\bfY{\mathbf{Y}}
\def\bfZ{\mathbf{Z}}

\def\uC{\mathrm{C}}
\def\uD{\mathrm{D}}

\def\uG{\mathrm{G}}

\def\uK{\mathrm{K}}
\def\uL{\mathrm{L}}
\def\uM{\mathrm{M}}
\def\uN{\mathrm{N}}

\def\uX{\mathrm{X}}

\newcommand{\N}{\mathbb{N}}
\newcommand{\Z}{\mathbb{Z}}

\newcommand{\sdif}{\triangle}      

\newcommand{\Bigcap}[2][\relax]{%
 \ifx#1\relax \bigcap_{#2}
 \else \bigcap^{#1}_{#2}
 \fi}
\newcommand{\Bigcup}[2][\relax]{%
 \ifx#1\relax \bigcup_{#2}
 \else \bigcup^{#1}_{#2}
 \fi}

\newcommand{\Sup}{\,\vee\,}
\newcommand{\Inf}{\,\wedge\,}

\def\Homeo{\mathrm{Homeo}}              

\DeclareMathOperator{\Hom}{Hom}

\DeclareMathOperator{\Aut}{Aut}

\def\fact#1#2{#1/#2}
\def\tfact#1#2{#1/#2}

\def\fact#1#2{{\raise0.2em\hbox{$#1$}\kern-0.2em/\kern-0.1em\lower0.2em\hbox{$#2$}}}
\def\tfact#1#2{{\raise0.1em\hbox{\small$#1$}\kern-0.1em/\kern-0.1em\lower0.1em\hbox{\small$#2$}}}

\newcommand{\norm}[2][\relax]{
   \ifx#1\relax \ensuremath{\left\Vert#2\right\Vert}
   \else \ensuremath{\left\Vert#2\right\Vert_{#1}}
   \fi}

\newcommand{\Bnorm}[2][\relax]{
   \ifx#1\relax \ensuremath{\Bigl\Vert#2\Bigr\Vert}
   \else \ensuremath{\Bigl\Vert#2\Bigr\Vert_{#1}}
   \fi}

\makeatletter
\newcommand{\tdprod}[2]{\ensuremath{%
  \setbox0=\hbox{\ensuremath{\langle#1,#2 \rangle}}
  \dimen@\ht0
  \advance\dimen@ by \dp0 (#1\rule[-\dp0]{0pt}{\dimen@}\,|#2\hspace{1pt})}}
\newcommand{\dprod}[2]{\ensuremath{%
  \setbox0=\hbox{\ensuremath{\left\langle#1,#2\right\rangle}}
  \dimen@\ht0
  \advance\dimen@ by \dp0 \left\langle\left.#1\rule[-\dp0]{0pt}{\dimen@}\,\right|#2\hspace{1pt}\right\rangle}}

\newcommand{\bdprod}[2]{\ensuremath{%
  \setbox0=\hbox{\ensuremath{\bigl\langle#1,#2\bigr\rangle}}
  \dimen@\ht0
  \advance\dimen@ by \dp0 \bigl\langle#1\bigl|\rule[-\dp0]{0pt}{\dimen@}\bigr.#2\hspace{1pt}\bigr\rangle}}
\newcommand{\Bdprod}[2]{\ensuremath{%
  \setbox0=\hbox{\ensuremath{\Bigl\langle#1,#2\Bigr\rangle}}
  \dimen@\ht0
  \advance\dimen@ by \dp0 \Bigl\langle#1\Bigl|\rule[-\dp0]{0pt}{\dimen@}\Bigr.#2\hspace{1pt}\Bigr\rangle}}
\newcommand{\tsprod}[2]{\ensuremath{%
  \setbox0=\hbox{\ensuremath{(#1,#2)}}
  \dimen@\ht0
  \advance\dimen@ by \dp0 (#1\rule[-\dp0]{0pt}{\dimen@}\,|#2\hspace{1pt})}}
\newcommand{\sprod}[2]{\ensuremath{%
  \setbox0=\hbox{\ensuremath{\left(#1,#2\right)}}
  \dimen@\ht0
  \advance\dimen@ by \dp0 \left(\left.#1\rule[-\dp0]{0pt}{\dimen@}\,\right|#2\hspace{1pt}\right)}}
\newcommand{\bsprod}[2]{\ensuremath{%
  \setbox0=\hbox{\ensuremath{\bigl(#1,#2\bigr)}}
  \dimen@\ht0
  \advance\dimen@ by \dp0 \bigl(#1\bigl|\rule[-\dp0]{0pt}{\dimen@}\bigr.#2\hspace{1pt}\bigr)}}
\newcommand{\Bsprod}[2]{\ensuremath{%
  \setbox0=\hbox{\ensuremath{\Bigl(#1,#2\Bigr)}}
  \dimen@\ht0
  \advance\dimen@ by \dp0 \Bigl(#1\Bigl|\rule[-\dp0]{0pt}{\dimen@}\Bigr.#2\hspace{1pt}\Bigr)}}
\makeatother

\newcommand{\Ce}{\mathrm{C}}

\newcommand{\Ell}[2][\relax]{
   \ifx#1\relax \mathrm{L}^{\mathrm{#2}}
   \else \mathrm{L}^{\mathrm{#2}}_{\mathrm{#1}}
   \fi}
\renewcommand{\Ell}[2][\relax]{
   \ifx#1\relax \mathrm{L}^{\!#2}
   \else \mathrm{L}^{\!#2}_{\mathrm{#1}}
   \fi}

\newcommand{\Wee}[2][\relax]{
   \ifx#1\relax \mathrm{W}^{\mathrm{#2}}
   \else \mathrm{W}^{\mathrm{#2}}_{\mathrm{#1}}
   \fi}
\newcommand{\Har}[2][\relax]{
   \ifx#1\relax \mathsf{H}^{\mathsf{#2}}
   \else   \mathsf{H}^{\mathsf{#2}}_{\mathrm{#1}}
   \fi}

\def\prX{\mathrm X}    

\def\rlqed{\rlap{\rule{\hsize}{0pt}\kern-1ex\kern-1em\qed}}

\makeatletter

\def\maketag@@@@@#1{\llap{\hbox to\hsize{\m@th\normalfont#1}}%
\gdef\tagform@##1{\maketag@@@{(\ignorespaces##1\unskip\@@italiccorr)}}}

\def\eqtext#1{\gdef\tagform@##1{\maketag@@@@@{\ignorespaces##1\unskip\@@italiccorr\hfill}}\tag{#1}}%
\def\reqtext#1{\gdef\tagform@##1{\maketag@@@@@{\hfill\ignorespaces##1\unskip\@@italiccorr}}\tag{#1}}%
\def\leqtext#1{\gdef\tagform@##1{\maketag@@@@@{\ignorespaces##1\unskip\@@italiccorr}}\tag{#1}}%
\makeatother

\newcommand{\SET}{\mathsf{Set}}
\newcommand{\FPow}{\scrP_0}
\newcommand{\Powa}{\scrP}
\newcommand{\SLPow}[1]{\mathrm{Lat}(#1)}
\newcommand{\SLPowg}[2]{\mathrm{Lat}_{#1}(#2)}
\newcommand{\SLLPow}[1]{\mathrm{LocLat}(#1)}
\newcommand{\SLLPowg}[2]{\mathrm{LocLat}_{#1}(#2)}
\newcommand{\Sub}[2]{\mathsf{Sub}_{#1}(#2)}
\newcommand{\Suba}[1]{\mathsf{Sub}(#1)}

\newcommand{\SSubm}[3]{\mathsf{SpSub}_{#1}^{#3}(#2)}
\newcommand{\SSubam}[2]{\mathsf{SpSub}^{#2}(#1)}
\newcommand{\Rp}{\mathbb{R}_{\geq 0}}

\newcommand{\MDL}{\mathsf{ML}}
\newcommand{\OBS}{\mathsf{Obs}}
\newcommand{\PSYS}{\mathsf{PhysSys}}
\newcommand{\CVAL}{\mathsf{CompVal}}

\newcommand{\Rep}[2]{\mathsf{Rep}_{#1}^{#2}}

\newcommand{\TOP}{\mathsf{Top}}
\newcommand{\PSP}{\mathsf{Prob}}
\newcommand{\op}{\mathrm{op}}
\newcommand{\KS}{\mathrm{KS}}
\newcommand{\ev}{\mathrm{ev}}
\newcommand{\eqdef}{\overset{\mathrm{def}}{=\joinrel=}}
\DeclareMathOperator{\ob}{Obj}
\date{\today}

\begin{document}

\title[Measurement functors and entropy]{Categories of Measurement Functors. \\  Entropy of Discrete Amenable Group \\ Representations on Abstract Categories. \\
Entropy as a Bifunctor into $[0,\infty]$.}

\author[Nikita Moriakov]{Nikita Moriakov}
\address{Delft Institute of Applied Mathematics, Delft University of Technology,
P.O. Box 5031, 2600 GA Delft, The Netherlands}

\email{n.moriakov@tudelft.nl}

\subjclass{Primary  28D20, 37A35, 37B40; Secondary 94A17, 18A99, 18F99}
\renewcommand{\subjclassname}{\textup{2000} Mathematics Subject
    Classification}

\thanks{The author kindly acknowledges the support from ESA CICAT of TU Delft}
\date{\today}

\begin{abstract}
The main purpose of this article is to provide a common generalization of the notions of a topological and Kolmogorov-Sinai entropy for arbitrary representations of discrete amenable groups on  objects of (abstract) categories. This is performed by introducing the notion of a measurement functor from the category of representations of a fixed amenable group $\Gamma$ on objects of an abstract category $\mathrm{C}$ to the category of representations of $\Gamma$ on distributive lattices with localization. We develop the entropy theory of representations of $\Gamma$ on these lattices, and then define the entropy of a representation of $\Gamma$ on objects of the category $\mathrm{C}$ with respect to a given measurement functor. For a fixed measurement functor, this entropy decreases along arrows of the category of representations. For a fixed category, entropies defined via different measurement functors decrease pointwise along natural transformations of measurement functors. We conclude that entropy is a bifunctor to the poset of extended positive reals. As an application of the theory, we show that both topological and Kolmogorov-Sinai entropies are instances of entropies arising from certain measurement functors.
\end{abstract}

\maketitle

\section{Introduction}
\label{s.intro}

The first results unifying topological and Kolmogorov-Sinai entropy theories for $\Z$-actions were obtained by G\"{u}nther Palm in his PhD dissertation. The goal of his research was finding a generalization of these theories in the language of functional analysis. We review his approach briefly. To every topological dynamical system on a topological space $\uK$ one can associate the corresponding Koopman representation on the Banach lattice $\Ce(\uK)$ of continuous functions. Similarly, to every measure-preserving dynamical system on a probability space $\prX$ one can associate the corresponding Koopman representation on the Banach lattice $\Ell{1}(\prX)$ of integrable functions. These classical lattices are nowadays called \emph{Banach lattices of observables.} Taking the set of closed Banach lattice ideals of these lattices, one retrieves the collection of closed and measurable sets respectively, and also the dynamics on the underlying spaces by using some appropriate duality theorems. Both types of lattices are examples of Banach lattices with quasi-interior points\footnote{See \cite{schaefer} for the definitions.}. It is well-known that the set of closed ideals of a Banach lattice is in fact a distributive lattice under some natural lattice operations. It is easy to see that an action of $\Z$ on a Banach lattice with quasi-interior point induces an action on its distributive lattice of closed ideals. For this action, Palm introduced a concept of entropy that coincides with the classical notions of entropy when the underlying Banach lattices with quasi-interior points are the classical lattices of observables. For the details we refer to \cite{palm}. In this work, however, we significantly deviate from the original approach for the following reasons.

Firstly, we have not seen much usage of Banach lattices with quasi-interior points in the theory of dynamical systems so far. It is not clear why these lattices should be the proper setup for a `very general entropy theory', since they appear to be a pure functional analysis phenomenon. Unfortunately, representation results such as the Kakutani representation theorem \cite{schaefer} do not answer all the relevant questions. For instance, the entropy is defined by looking at the induced action on the distributive lattice of closed lattice ideals. However, the structure of this lattice of ideals is in general not very well described through representation results. Thus the structure of the lattices of open (measurable) sets becomes substantially more obscure when we pass to the lattice of observables.

Secondly, in both topological and Kolmogorov-Sinai entropy theories, the entropy of a system is greater or equal than the entropy of its factors, i.e. entropy decreases along the arrows of the category. However, we were not able to prove this statement in the setting of Palm's theory - there is no notion of a factor to start with - and there is a `structural' counterexample at the level of distributive lattices that we shall provide in Section \ref{ss.structcex}. Hence, the structure of an abstract distributive lattice is not sufficient, and the additional structure that we see when considering lattices of closed ideals of Banach lattices with quasi-interior points is obscure. Therefore, we chose to shift the focus away from Banach lattices and search for a way to impose additional structure on lattices via an explicit `functorial' correspondence instead.

Finally, the theories of entropy for general discrete amenable group actions based on the Ornstein-Weiss lemma were not even developed at the time of Palm's work.

These considerations have lead to the approach described in this work. Our goal is to develop a concept of entropy for representations of discrete amenable groups on arbitrary categories that would
\begin{aufziii}
\item be functorial in nature, i.e. would give entropy of an arbitrary representation $\pi$ on object $\bfA$ of an abstract category $\uC$, once a \emph{measurement functor}, describing some underlying structure of a distributive lattice on $\bfA$, is supplied;

\item decrease along arrows of the category of representations as in the classical setting;

\item behave `reasonable' when one has two measurement functors with an `arrow' between them;

\item generalize the classical notions of entropy for discrete amenable group actions.
\end{aufziii}

These goals have been accomplished in this work. The innermost core of the proposed theory is a modified (localized) notion of Palm's entropy for measured distributive lattices, and the surrounding body of categorical machinery is built to resolve the issues mentioned above and to make the core concept work best within modern mathematical frameworks.

We will transfer the structure from our categories of dynamical systems to the category of representations on \emph{measured distributive lattices with localization} (see Section \ref{ss.dynlatdef}) via a functorial correspondence, thus we have to prepare some categorical language. Let $\Gamma$ be an arbitrary group, and let $\uC$ be some category\footnote{Unless stated otherwise, all categories are assumed to be locally small \cite{maclane}.}. The \textbf{category of representations of $\Gamma$ on} $\uC$ is the category $\Rep{\Gamma}{\uC}$ with $\ob(\Rep{\Gamma}{\uC})$ being the collection of all pairs $(\bfA, \pi)$ where $\bfA \in \ob(\uC)$ and $\pi: \Gamma \to \Aut(\bfA)$ is a group homomorphism. For any objects $(\bfA,\pi)$, $(\bfB,\rho)$ in $\ob(\Rep{\Gamma}{\uC})$ we define the set of morphisms $\Hom((\bfA,\pi),(\bfB,\rho))$ as those $\psi \in \Hom(\bfA,\bfB)$ s.t. for every $\gamma \in \Gamma$ the diagram
\begin{equation*}
 \xymatrix{ \bfA \ar[r]^{\psi} & \bfB  \\
               \bfA \ar[u]^{\pi_{\gamma}} \ar[r]_{\psi} & \bfB \ar[u]_{\rho_{\gamma}} }
\end{equation*}
commutes. Then $\bfB$ is called a \textbf{factor} of $\bfA$. We do not require $\psi$ above to be an epimorphism in order to comply with the established categorical terminology in dynamical systems. When it is clear what are the representations, we will use the notation $\Hom_{\Gamma}(\bfA,\bfB)$ instead of $\Hom((\bfA,\pi),(\bfB,\rho))$ for the set of $\Gamma$-equivariant morphisms.

Taking measurements\footnote{See the introductory part of \cite{gromovi} for a motivation of this viewpoint in `physical' systems.} can be seen as a covariant \emph{measurement functor} $\uM$ from the category $\Rep{\Gamma}{\uC}$ to the category $\Rep{\Gamma}{\MDL}$ of representations on a certain category of distributive lattices that we will introduce in Section \ref{ss.dynlatdef}. On this category, as we will see in this article, one can define entropy of a group action for discrete amenable $\Gamma$. That is, for every $(\bfX,\pi) \in \Rep{\Gamma}{\MDL}$ we define some nonnegative extended real value $h_{\MDL}(\bfX,\pi)$. Thus we can define entropy on $\Rep{\Gamma}{\uC}$ with respect to a measurement functor $\uM$ by
\begin{equation}
\label{eq.entfunc}
h((\bfA,\pi),\uM):=h_{\MDL}(\uM(\bfA,\pi)).
\end{equation}

In general it is not clear if for a given category of representations an interesting measurement functor exists at all, but we do expect that it exists in topological and measure-theoretic dynamics, where one takes distributive lattices of open/measurable subsets for the distributive lattice respectively. So we will provide
a construction of measurement functors $\uM_{\PSP}$ and $\uM_{\TOP}$ that yield the classical notions of entropy: Kolmogorov-Sinai entropy and topological entropy respectively.

We will now briefly outline the structure of the article. In Section \ref{s.adlmeas} we introduce the key notions of a measured lattice with localization, an abstract dynamical lattice and a measurement functor. At the end of the section we provide the construction of the measurement functors $\uM_{\TOP}$ and $\uM_{\PSP}$.

We begin Section \ref{s.entropy} by recalling the notion of an amenable group and stating the Ornstein-Weiss lemma, Proposition \ref{p.ornsteinweiss}. This lemma allows to introduce the entropy $h_{\MDL}$ of representations of discrete amenable groups on measured lattices with localization.

We then define entropy of a representation on a category with respect to the given measurement functor, and show that this entropy decreases along morphisms in $\Rep{\Gamma}{\uC}$. We remind the reader of the notion of a \emph{natural transformation} between functors, and also prove that the entropy decreases pointwise on $\Rep{\Gamma}{\uC}$ along natural transformations of measurement functors. We introduce the category of measurement functors $[\Rep{\Gamma}{\uC},\Rep{\Gamma}{\MDL}]_{\frakM}$, which allows to view evaluation $\ev: (\bfA,\uM) \mapsto \uM(\bfA)$ as a bifunctor from the product category $\Rep{\Gamma}{\uC} \times [\Rep{\Gamma}{\uC},\Rep{\Gamma}{\MDL}]_{\frakM}$ to $\Rep{\Gamma}{\MDL}$. We conclude in Corollary \ref{cor.entropbifunc} that the entropy defined in equation (\ref{eq.entfunc}) is also a bifunctor from the corresponding product category to the poset of reals, i.e. we have
\begin{equation*}
\Rep{\Gamma}{\uC} \times [\Rep{\Gamma}{\uC},\Rep{\Gamma}{\MDL}]_{\frakM} \overset{h(\cdot,\cdot)}{\longrightarrow} [0, \infty].
\end{equation*}

Finally, we prove that the entropy defined with the help of measurement functors from Section \ref{s.adlmeas} simultaneously generalizes topological and Kolmogorov-Sinai entropy.

I would like to thank my advisor Markus Haase for suggesting to look into Palm's work and for the numerous invaluable comments and corrections he has provided.

\section{Abstract dynamical lattices and measurement functors}
\label{s.adlmeas}
\subsection{Category of lattices with localization}
\label{ss.dynlatdef}
For convenience we use the term \textbf{distributive lattice} for distributive lattices with $0$ (`bottom') and $1$ (`top') s.t. $0 \neq 1$. The set of all finite subsets of a set $\uX$ is denoted by $\FPow(\uX)$. When talking about lattice embeddings of distributive lattices, we assume that these embeddings respect the top and bottom elements. Let $V$ be a distributive lattice. A finite subset $\alpha \in \FPow(V)$ is called a \textbf{cover} if $\sup \alpha = 1$. Clearly, lattice embeddings map covers to covers. The set of all covers of $V$ is denoted by $\Sigma_V$. The set of all distributive sublattices of $V$ is denoted by $\SLPow{V}$. The set of all distributive sublattices of $V$ containing a given cover $\alpha$ is denoted by $\SLPowg{\alpha}{V}$.

We will now prepare the key ingredients for the definition of the category of measured lattices with localization. A function $m: V \to \Rp$ is called a \textbf{measurement function} if it satisfies the conditions that
\begin{aufzi}
\item $m(0) = 0, m(1) \neq 0$
\item $m(a) = 0 \Rightarrow m(a \Sup b) = m(b)$ for all $b \in V$.
\end{aufzi}

Measurement functions tell us how big or `likely' the elements of $V$ are. A function $\Omega: \Sigma_V \to \SLPow{V}$ is called a \textbf{localization function} if for every cover $\alpha \in \Sigma_V$ the sublattice $\Omega(\alpha) \subseteq V$ contains $\alpha$.  Later we will think of $\Omega(\alpha)$ as `the smallest subsystem which realizes $\alpha$'.

The category $\MDL$ of \textbf{measured lattices with localization} is defined as follows. Objects of the category $\MDL$ are all triples $(V,m, \Omega)$, where

\begin{aufzi}
\item $V$ is a distributive lattice;

\item $m: V \to \Rp$ is a measurement function;

\item $\Omega: \Sigma_V \to \SLPow{V}$ is a localization function.
\end{aufzi}
To complete the definition of the category of measured lattices with localization we also need to say what the arrows are. Let $\bfV=(V,m_1,\Omega_1),\bfW=(W,m_2,\Omega_2)$ be a pair of distributive lattices with localization. Let $\Phi: W \to V$ be a lattice embedding such that
\begin{aufzi}
\item $\Phi$ preserves the measurement function, i.e. $m_1(\Phi(a)) = m_2(a)$ for all elements $a \in W$;

\item $\Phi$ preserves the localization function, i.e. $\Omega_1(\Phi(\alpha)) = \Phi (\Omega_2(\alpha))$ for all covers $\alpha \in \Sigma_W$.
\end{aufzi}

Then we call $\Phi^{\op}$ a morphism between $\bfV$ and $\bfW$ and write $\Hom(\bfV, \bfW)$ for the collection of all morphisms obtained this way. The superscript `$\op$' in $\Phi^{\op}$ above indicates that even though $\Phi$ is a mapping from $W$ to $V$, the morphism of lattices determined by $\Phi$ points in the opposite direction. We will typically define morphisms by defining corresponding lattice embeddings, so we write `$\op$' to avoid confusion about direction of morphism. We have chosen to `switch the arrows' so that the arrows in $\MDL$ go in the same direction as the arrows in some category $\uC$ that is being `represented' on $\MDL$. It will be clear later on what we precisely mean.

Let $W \subseteq V$ be a sublattice of a measured distributive lattice with localization $(V,m,\Omega)$. We call $W$ a \textbf{local sublattice} of the lattice $V$ if $\Omega(\Sigma_W) \subseteq \SLPow{W}$, that is, if for every cover $\alpha \in \Sigma_W$ the corresponding localization $\Omega(\alpha)$ is actually a sublattice of $W \subseteq V$. Hence if $W \subseteq V$ is a local sublattice, we obtain a measured lattice with localization $(W,m,\Omega)$ from $(V,m,\Omega)$ by restricting the functions $m$ and $\Omega$. We denote by $\SLLPow{\bfV}$ the set of local sublattices of $\bfV=(V,m,\Omega)$, and by $\SLLPowg{\alpha}{\bfV}$ the set of local sublattices containing a cover $\alpha \in \Sigma_V$. Both of these sets are nonempty, since they contain at least $V$. Intuitively speaking, we are only interested in local sublattices of a lattice $V$ because they correspond to `subsystems' in the category $\MDL$. If $\bfV=(V,m,\Omega)$ is a measured lattice with localization, we will often write $\Sigma_{\bfV}$ to denote the set $\Sigma_V$ of covers of $V$.

We now provide the classical examples of lattices with localization. These are not the simplest examples possible, but are the most important ones for the purposes of this work. We will also return to these examples when discussing measurement functors at the end of Section \ref{ss.measfunctd}.

\subsubsection{Topological lattices}
\label{ex.toplatint}
Consider the category $\TOP$ of nonempty compact Hausdorff spaces with surjective continuous maps as morphisms. Objects of this category are the pairs $\bfA = (\uK,\calU)$, where $\calU$ is a compact Hausdorff topology on a set $\uK$. The nonempty collection of open sets $\calU$ is a distributive lattice under the operations of set union and intersection. The measured lattice with localization $(\calU,m,\Omega)$ associated with $\bfA$ can now be introduced. The measurement function $m$ is defined by
\begin{equation*}
m(A) := \begin{cases}
    1 & \text{if } A \text{ is a nonempty;}\\
    0 &  \text{otherwise.}
  \end{cases}
\end{equation*}
The localization function $\Omega$ maps an open cover $\alpha \in \Sigma_{\calU}$ to the smallest topology generated by $\alpha$. We denote by $\overline \bfA$ the lattice $(\calU,m,\Omega)$ obtained this way. It is easy to see that every morphism $\phi: \bfA \to \bfB$ in $\TOP$ induces a morphism $\overline \phi: \overline \bfA \to \overline \bfB$ by taking preimages of open sets. Hence the correspondence $\bfA \to \overline \bfA$ is in fact a covariant functor $\uM_{\TOP}^0$ between $\TOP$ and $\MDL$.

\subsubsection{Measure lattices}
\label{ex.measlatint}
Now we consider the category $\PSP$ of standard probability spaces with equivalence classes of measure-preserving maps as morphisms.  Objects of this category are triples $\bfX = (\prX,\calM,\mu)$, where $\calM$ is the measure algebra of a standard probability space $\prX$. The nonempty collection of equivalence classes of measurable sets $\calM$ is a distributive lattice under operations of union and intersection. We introduce the measured lattice with localization $(\calM,m,\Omega)$, associated with $\bfX$. The measurement function is defined by $m(a): = \mu(a)$ for all $a \in \calM$. The localization function $\Omega$ takes a cover $\alpha \in \Sigma_{\calM}$ and maps it to the smallest measure algebra generated by $\alpha$. We denote by $\overline \bfX$ the lattice $(\calM,m,\Omega)$ obtained this way. Every morphism $\phi: \bfX \to \bfY$ in $\PSP$ induces a morphism $\overline \phi: \overline \bfX \to \overline \bfY$ by taking preimages of measurable sets. Hence the correspondence $\bfX \to \overline \bfX$ is in fact a covariant functor $\uM_{\PSP}^0$ between $\PSP$ and $\MDL$.

\subsection{Representations on the category of lattices with localization}
Now consider the associated category of representations $\Rep{\Gamma}{\MDL}$, and call it the category of \textbf{abstract dynamical lattices}. The objects of $\Rep{\Gamma}{\MDL}$ are pairs $(\bfV,\pi)$ of a measured distributive lattice with localization $\bfV = (V,m,\Omega)$ and a representation $\pi$ of $\Gamma$ in $\Aut(\bfV)$. To simplify the notation we will often write $(V,m,\Omega; \pi)$ in place of $((V,m,\Omega),\pi)$ to denote abstract dynamical lattices. Given $(\bfV,\pi) \in \ob(\Rep{\Gamma}{\MDL})$ we write $\SLLPow{\bfV,\pi}$ to denote the set of all $\Gamma$-invariant local sublattices $W$ of $V$, and similarly we write $\SLLPowg{\alpha}{\bfV,\pi}$ to denote the set of all local $\Gamma$-invariant sublattices of $V$ containing a cover $\alpha \in \Sigma_V$. Every $\Gamma$-invariant local sublattice $W \subseteq V$ yields an abstract dynamical lattice $\bfW = (W,m,\Omega; \rho)$, where we take $\rho$ to be the restriction of the representation $\pi$ to $W$. If $\imath: W \to V$ is the inclusion mapping for lattices $W$ and $V$ as above, then $\imath^{\op} \in \Hom_{\Gamma}(\bfV,\bfW)$ is a factor map. Conversely, every factor $\bfW$ of $\bfV$ corresponds to a $\Gamma$-invariant local sublattice of $V$, but factors isomorphic in the category $\Rep{\Gamma}{\MDL}$ can of course have distinct corresponding local sublattices. We will now make this relation between sublattices and factors more precise by introducing some language from category theory.

The sets of $\Gamma$-invariant local sublattices $\SLLPow{\bfV,\pi}$ and $\SLLPowg{\alpha}{\bfV,\pi}$ have canonical structures of small\footnote{That is, the collection of objects is in fact a set.} categories. We will now explain this structure for $\SLLPowg{\alpha}{\bfV,\pi}$ and define the category $\Sub{\alpha}{\bfV,\pi}$ (with obvious modifications in the definition of $\Suba{\bfV,\pi}$). The set of objects $\ob(\Sub{\alpha}{\bfV,\pi})$ is by definition $\SLLPowg{\alpha}{\bfV,\pi}$. For two objects $\bfX,\bfY \in \ob(\Sub{\alpha}{\bfV,\pi})$ with corresponding $\Gamma$-invariant local sublattices $X, Y \subseteq V$ we set
\begin{equation*}
\Hom(\bfX,\bfY) := \begin{cases}
    \imath^{\op} & \text{if } Y  \overset{\imath}{\subseteq} X \text{ is a sublattice}\\
    \varnothing &  \text{otherwise.}
  \end{cases}
\end{equation*}

So in particular for any two objects $\bfX,\bfY \in \ob(\Sub{\alpha}{\bfV,\pi})$ the set of morphisms consists of at most one arrow, and it might be empty. Thus a final object in a subcategory $\uD$ of $\Sub{\alpha}{\bfV,\pi}$ is a $\Gamma$-invariant local sublattice of $V$ containing $\alpha$ that is contained in any other lattice in $\uD$. Furthermore, there can be only one final object. Category $\Sub{\alpha}{\bfV,\pi}$ is a full subcategory of $\Suba{\bfV,\pi}$.

So far we have given a `concrete' definition of the category $\Suba{\bfV,\pi}$ of local $\Gamma$-invariant sublattices, but it is in fact an instance of a general category-theoretic notion of \textbf{category of subobjects}\footnote{This should not be confused with the notion of a subobject poset, the difference is that we do not require objects to be monomorphisms.}. For the moment, let $\uC$ be a category and let $\bfA$ be a fixed object of $\uC$. We define the category $\Suba{\bfA}$ as follows. Let $S$ be the collection of all arrows $\phi: \bfA \to \bfB$ in $\uC$. Given arrows $\phi: \bfA \to \bfB$ and $\psi: \bfA \to \bfC$ in $S$, we let $\Hom(\phi,\psi)$ be the set of morphisms $\zeta: \bfB \to \bfC$ in $\uC$ such that $\zeta \circ \phi = \psi$. Finally, we define $\ob(\Suba{\bfA})$ as the collection of isomorphism classes of elements of $S$ modulo isomorphisms in $\Hom$ we defined above, and for $[\phi],[\psi] \in \ob(\Suba{\bfA})$ we take $\Hom([\phi],[\psi])$ to be the corresponding set\footnote{Note that this is indeed a set since $\uC$ is a locally small category.} of equivalence classes of morphisms in $\Hom(\phi,\psi)$.

It is straightforward to check that the category of sublattices $\Suba{\bfV,\pi}$ we have defined for distributive lattices is in fact the category of subobjects of $(\bfV,\pi)$. The functor $\uG$ giving this category equivalence takes an arrow $$\imath^{\op}: (V,m_V,\Omega_V; \pi) \to (W,m_W,\Omega_W; \rho)$$ and maps it to the $\Gamma$-invariant local sublattice of $V$ determined by the range of the lattice embedding $\imath$. Equivalent arrows have identical range, and morphisms between arrows define inclusions between the ranges. Hence there is actually no ambiguity in the notation for subobject. Furthermore, this shows that the collection of objects in $\Suba{\bfV,\pi}$ is in fact a set. This correspondence between arrows and sublattices is a generalization of the well-known correspondence between factors of a given measure-preserving dynamical system $\bfX$ and coordinatizations of these factors as subalgebras in ergodic theory.

We also need to remind the reader of the notion of a \textbf{coproduct} from category theory. It is a `dual' notion to the notion of product. Let $\uC$ be a category, and let $\{ \bfX_{\lambda} \}_{\lambda \in \Lambda}$ be a family of objects of this category indexed by a set $\Lambda$. The coproduct of this family is a pair of an object $\bfX \in \uC$ and a collection of morphisms $\{ \pi_{\lambda} \}_{\lambda \in \Lambda}$, $\pi_{\lambda}: \bfX_{\lambda} \to \bfX$ such that for any object $\bfY \in \uC$ and a collection of morphisms $\{ \rho_{\lambda }\}_{\lambda \in \Lambda}$, $\rho_{\lambda}: \bfX_{\lambda} \to \bfY$ there exists a unique morphism $\phi: \bfX \to \bfY$ s.t. $\phi \circ \pi_{\lambda} = \rho_{\lambda}$ for all indices $\lambda \in \Lambda$. A standard argument then shows that coproducts are unique up to a unique isomorphism when they exist. We write $\coprod\limits_{\lambda \in \Lambda} \bfX_{\lambda}$ for the coproduct of the family $\{ \bfX_{\lambda} \}_{\lambda \in \Lambda}$.

We give a few examples of subobject categories and coproducts now. Very often coproducts amount to taking `disjoint union' of structures, but as will see shortly they look different when working with topological and measure-theoretic factors, i.e. when we are taking coproducts inside respective categories of subobjects that we introduced above.

\subsubsection{Coproducts of sets, abelian groups, vector spaces}
Consider the category $\SET$ of sets with maps between sets being the morphisms. Then the coproduct of a family $\{ \bfX_{\lambda} \}_{\lambda \in \Lambda}$ of sets is the disjoint union $\bigsqcup\limits_{\lambda \in \Lambda} \bfX_{\lambda}$ of these sets, and maps $\pi_{\lambda}: \bfX_{\lambda} \to \bigsqcup\limits_{\lambda \in \Lambda} \bfX_{\lambda}$ are the canonical set inclusions. Verifying the universal property is straightforward. It is also not difficult to show that in the categories $\mathsf{Ab}$ of abelian groups and $\mathsf{Vect}_k$ of vector spaces over a fixed field $k$ coproduct is actually a direct sum.

\subsubsection{Products and coproducts of local sublattices} Let $(\bfV,\pi)$ be an abstract dynamical lattice and $\Suba{\bfV,\pi}$ be the category of its $\Gamma$-invariant local sublattices. Firstly, we have a look at the products in $\Suba{\bfV,\pi}$. Let $\bfX, \bfY$ be $\Gamma$-invariant local sublattices of $\bfV$. Let
\begin{equation*}
\bfZ:=\bigcap\{ \bfW: \bfW \text{ is a local invariant sublattice of } \bfV \text{ containing } \bfX,\bfY \}.
\end{equation*}

An intersection of any family of invariant local sublattices is local and invariant as well, and checking the universal property is straightforward. Hence $\bfZ = \bfX \prod \bfY$. Secondly, let $\{ \bfX_{\lambda} \}_{\lambda \in \Lambda}$ be a family of local invariant sublattices of $\bfV$. Let
\begin{equation*}
\bfW:=\bigcap\limits_{\lambda \in \Lambda} \bfX_{\lambda}.
\end{equation*}
Once again $\bfW$ is a local invariant sublattice and the universal property is also satisfied. Hence $\bfW = \coprod\limits_{\lambda \in \Lambda} \bfX_{\lambda}$.

\subsubsection{Topological factors and their coproducts}
\label{sss.coptop}
Let $\TOP$ be the category of compact Hausdorff topological spaces, introduced in Section \ref{ex.toplatint}. Let $\bfA \in \TOP$ be a fixed topological space, and let $\Suba{\bfA}$ be the associated category of subobjects. If $\bfX$ and $\bfY$ are factors of $\bfA$ s.t. the corresponding subtopologies on $\bfA$ coincide, then $\bfX$ equals $\bfY$ in $\Suba{\bfA}$. This shows that the functor $\uM_{\TOP}^0$ from Section \ref{ex.toplatint} induces a covariant embedding $\widetilde \uM_{\TOP}^0: \Suba{\bfA} \to \Suba{\overline \bfA}$\footnote{Here $\overline \bfA$ is the distributive lattice with localization introduced in Section \ref{ex.toplatint}.}. The range of this functor is a full subcategory of $\Suba{\overline \bfA}$, and, furthermore, this shows that $\Suba{\bfA}$ is a small category.

We describe a coproduct of two factors first\footnote{Even though we have described a category embedding above, where the codomain category admits products and coproducts, we still have to perform this check.}. So let $\phi: \bfA \to \bfX$ and $\psi: \bfA \to \bfY$ be representatives of $[\phi], [\psi] \in \Suba{\bfA}$. Consider the factor $\bfZ$ of $\bfA$ obtained by
\begin{aufziii}
\item lifting the topologies of $\bfX$ and $\bfY$ to subtopologies on $\bfA$ via maps $\phi$ and $\psi$;

\item taking the intersection of these subtopologies, obtaining a new compact topology on $\bfA$;

\item gluing points that are not separated by this topology to define a quotient map $\chi: \bfA \to \bfZ$.
\end{aufziii}
Then $\bfZ \in \TOP$ and $[\chi] \in \Suba{\bfA}$ is the coproduct of $[\psi]$ and $[\phi]$.  Similar construction applies to infinite coproducts.

\subsubsection{Measure-theoretic factors and their coproducts}
\label{sss.copmeas}
Let $\PSP$ be the category of standard probability spaces, introduced in Section \ref{ex.measlatint}. Let $\bfA \in \PSP$ be a fixed probability space, and let $\Suba{\bfA}$ be the associated category of subobjects. If $\bfX$ and $\bfY$ are factors of $\bfA$ s.t. the corresponding subalgebras on $\bfA$ coincide, then $\bfX$ equals $\bfY$ in $\Suba{\bfA}$. This shows that the functor $\uM_{\PSP}^0$ from Section \ref{ex.measlatint} induces a covariant embedding $\widetilde \uM_{\PSP}^0: \Suba{\bfA} \to \Suba{\overline \bfA}$\footnote{Here $\overline \bfA$ is the distributive lattice with localization introduced in Section \ref{ex.measlatint}.}. The range of this functor is a full subcategory of $\Suba{\overline \bfA}$ as well\footnote{Because of the duality between embeddings of separable measure algebras and morphisms of correpsonding standard probability spaces, see Theorem 2.15 in \cite{glasner}.}, and $\Suba{\bfA}$ is a small category, too.

We begin by describing a coproduct of two factors first. So let $\phi: \bfA \to \bfX$ and $\psi: \bfA \to \bfY$ be representatives of $[\phi], [\psi] \in \Suba{\bfA}$ with corresponding measure subalgebras $\overline \bfX$ and $\overline \bfY$ of the measure algebra $\overline \bfA$ of $\bfA$. Consider the factor $\bfZ$ of $\bfA$ obtained by
\begin{aufziii}
\item intersecting the measure subalgebras $\overline \bfX$ and $\overline \bfY$, obtaining a measure subalgebra of $\overline \bfA$;

\item using duality Theorem 2.15 in \cite{glasner} to get a factor map $\chi: \bfA \to \bfZ$ with corresponding measure algebra $\overline \bfX \cap \overline \bfY$.
\end{aufziii}
Then $\bfZ \in \PSP$ and $[\chi] \in \Suba{\bfA}$ is the coproduct of $[\psi]$ and $[\phi]$.  Similar construction applies to infinite coproducts.

\subsection{Measurement functors}
\label{ss.measfunctd}
Now we can introduce the notion of a measurement functor. Let $\uC$ be a category and let $\Rep{\Gamma}{\uC}$ be the associated category of representations. A covariant functor $\uM: \Rep{\Gamma}{\uC} \to \Rep{\Gamma}{\MDL}$ is called a \textbf{measurement functor} if it satisfies the following condition for every object $(\bfA,\pi) \in \ob(\Rep{\Gamma}{\uC})$:
\begin{quote}
If $\uM(\bfA,\pi) = (\overline \bfA, \overline \pi) \in \ob(\Rep{\Gamma}{\MDL})$, then for every cover $\alpha \in \Sigma_{\overline \bfA}$ the corresponding localization $\Omega(\alpha)$ within the lattice $\overline \bfA$ is precisely the minimal sublattice containing $\alpha$ that can be realized by the functor.
\end{quote}

We explain what these requirements mean. The objects of the category $\Rep{\Gamma}{\uC}$ are pairs $(\bfA, \pi)$, where $\bfA \in \ob(\uC)$ and $\pi: \Gamma \to \Aut(\bfA)$ is a group homomorphism. So, for every object $(\bfA, \pi) \in \ob(\Rep{\Gamma}{\uC})$ applying the functor $\uM$ yields a dynamical lattice
\begin{equation*}
\uM(\bfA, \pi) = (\overline{\bfA}, \overline{\pi}) \in \ob(\Rep{\Gamma}{\MDL}).
\end{equation*}
Here $\overline{\bfA} \in \ob(\MDL)$ is a measured lattice with localization, and $\overline{\pi}$ is a representation of $\Gamma$ in the group $\Aut(\overline{\bfA})$.

Morphisms in $\Hom((\bfA,\pi),(\bfB,\rho))$ are those morphisms $\psi$ in $\Hom(\bfA,\bfB)$ that satisfy $\psi \circ \pi_{\gamma} = \rho_{\gamma} \circ \psi$ for all $\gamma \in \Gamma$. For every such $\psi \in \Hom_{\Gamma}(\bfA,\bfB)$ the functor $\uM$ defines a $\Gamma$-equivariant morphism $\uM(\psi): \overline{\bfA} \to \overline{\bfB}$. This means that $\uM(\psi) \circ \overline{\pi}_{\gamma} = \overline{\rho}_{\gamma} \circ M(\psi)$ for every $\gamma \in \Gamma$, i.e. the diagram
\begin{equation*}
 \xymatrix{ \overline{\bfA} \ar[r]^{\uM(\psi)} & \overline{\bfB}  \\
               \overline \bfA \ar[u]^{\overline{\pi}_{\gamma}} \ar[r]_{\uM(\psi)} & \overline{\bfB} \ar[u]_{\overline{\rho}_{\gamma}} }
\end{equation*}
commutes.

Now we explain the main condition, for that we need to introduce some language first. Let $\uM$ be a covariant functor as above, $(\bfA,\pi) \in \ob(\Rep{\Gamma}{\uC})$, $\uM(\bfA,\pi)=(\overline \bfA,\overline \pi)$ and let $\alpha \in \Sigma_{\overline{\bfA}}$ be a cover. Then $\overline{\bfA}$ is a measured lattice with localization, and $\Sub{\alpha}{\overline{\bfA}, \overline{\pi}}$ is the associated category of $\Gamma$-invariant local sublattices of $\overline{\bfA}$ containing $\alpha$. Some of these sublattices are coming from factors of the system $(\bfA,\pi)$ via applying $\uM$ - for example $\overline{\bfA}$ itself - while others may not. To make this precise, consider the categories $\Suba{\overline{\bfA},\overline{\pi}}$, $\Suba{\bfA,\pi}$ of subobjects of $\overline \bfA$ and $\bfA$ respectively. Observe that the functor $\uM$ induces a functor $\widetilde \uM: \Suba{\bfA,\pi} \to \Suba{\overline{\bfA},\overline \pi}$ in the following manner. Given an arrow $\phi: \bfA \to \bfB$ in $\Rep{\Gamma}{\uC}$, it is mapped to the equivalence class of arrows $[\uM(\phi)]$ in $\Suba{\overline \bfA, \overline \pi}$, and this definition is independent of the representative of $\phi$ in $[\phi] \in \Suba{\bfA,\pi}$. In this definition we are only using that the functor $\uM$ preserves commutative diagrams, in particular, that it preserves isomorphisms. We define the subcategory $$\SSubam{\overline \bfA, \overline{\pi}}{\uM} \subseteq \Suba{\overline{\bfA},\overline{\pi}}$$ of \textbf{spacial sublattices} w.r.t. the functor $\uM$ as the range of $\widetilde \uM$. In general this subcategory is not full, i.e. there can be $\Gamma$-equivariant lattice embeddings that are not spacial. Since the category $\Suba{\overline \bfA,\overline \pi}$ is in fact the category of local $\Gamma$-invariant sublattices of $\overline \bfA$ (via the identification that was explained at the end of Section \ref{ss.dynlatdef}), we can talk about the subcategory of spacial sublattices $$\SSubm{\alpha}{\overline{\bfA},\overline{\pi}}{\uM} \subseteq \SSubam{\overline{\bfA},\overline{\pi}}{\uM}$$ containing the cover $\alpha$, which is a full subcategory of $\SSubam{\overline{\bfA},\overline{\pi}}{\uM}$.

This language allows us to explain what the main requirement says. Namely, for every object $(\bfA,\pi) \in \ob(\Rep{\Gamma}{\uC})$ we consider its image $\uM(\bfA,\pi) = (\overline \bfA,\overline \pi)$, where $\overline \bfA=(A,m,\Omega)$ is a measured lattice with localization, and require that for every cover $\alpha \in \Sigma_{\overline \bfA}$ the lattice $\Omega(\alpha)$ is the final object in the category $\SSubm{\alpha}{\overline \bfA, \overline \pi}{\uM}$.

Hence $\Omega(\alpha)$ is the spacial sublattice canonically embedded in any other spacial sublattice containing $\alpha$ via a spacial (i.e. corresponding to a factor in $\Rep{\Gamma}{\uC}$) morphism. In particular, this requirement shows that $\Omega(\alpha)$ does depend on the representation $\pi$, since it is $\Gamma$-invariant w.r.t. $\overline \pi$.

\begin{exas}
We will now provide some trivial examples of spacial and non-spacial sublattices. For simplicity we will take $\Gamma$ to be trivial and work in the category $\PSP$. Let $\bfX = (\prX,\calM,\mu)$ be the probability space on the set $\prX := \{ a, b, c\}$ with $\calM:=\Powa(\prX)$ being the boolean algebra of all subsets of $\prX$ and $\mu$ be the uniform probability distribution. We consider the measured lattice with localization $\overline \bfX$ associated to $\bfX$ that was given in Section \ref{ex.measlatint}\footnote{This corresponds to the measurement functor $\uM_{\PSP}$ for trivial $\Gamma$, see Section \ref{sss.measdynex} below for more details.}.
Consider a sublattice $V:=\{ 0, \{a,b \}, \{ b\}, \{ b,c\} , 1 \}$ of $\calM$. Then $V$ is not a spacial sublattice because it is not a boolean algebra. Conversely, any subalgebra of $\calM$ defines a factor of $\bfX$ in $\PSP$ and thus is an example of a spacial sublattice.
\end{exas}

Given a covariant functor $\uM: \Rep{\Gamma}{\uC} \to \Rep{\Gamma}{\MDL}$, the main condition in the definition of a measurement functor is rather nontrivial, and it is not clear if the final object exists in the category $\SSubm{\alpha}{\overline \bfA, \overline \pi}{\uM}$. When it exists, we know that it is unique. We now provide an abstract condition that simplifies the verification of the second requirement.
\begin{prop}[Coproduct stability condition]
\label{p.coprodstab}
Let $\uM: \Rep{\Gamma}{\uC} \to \Rep{\Gamma}{\MDL}$ be a covariant functor. Let $(\bfA,\pi) \in \Rep{\Gamma}{\uC}$, $(\overline \bfA,\overline \pi)=\uM(\bfA, \pi)$, and let $\alpha \in \Sigma_{\overline \bfA}$. Suppose the following holds:
\begin{aufziii}
\item the category $\Suba{\bfA,\pi}$ is small and admits infinite coproducts;

\item if $\{ \bfX_{\lambda} \}_{\lambda \in \Lambda}$ is a collection of objects in $\Suba{\bfA,\pi}$ such that for every index $\lambda \in \Lambda$ object $\widetilde \uM(\bfX_{\lambda})$ is in $\SSubm{\alpha}{\overline \bfA, \overline \pi}{\uM}$, then the image of the coproduct $\widetilde \uM(\coprod\limits_{\lambda \in \Lambda} \bfX_{\lambda})$ is in $\SSubm{\alpha}{\overline \bfA, \overline \pi}{\uM}$ as well.
\end{aufziii}

Then $\widetilde \uM(\coprod\limits_{\lambda \in \Lambda} \bfX_{\lambda})$ is the final object in category $\SSubm{\alpha}{\overline \bfA, \overline \pi}{\uM}$, where $\{ \bfX_{\lambda} \}_{\lambda \in \Lambda}$ is the collection of all objects in $\Suba{\bfA,\pi}$ s.t. $\widetilde \uM(\bfX_{\lambda})$ is in $\SSubm{\alpha}{\overline \bfA, \overline \pi}{\uM}$ for all indices $\lambda$.
\end{prop}
\begin{proof}
The second condition implies that $\widetilde \uM(\coprod\limits_{\lambda \in \Lambda} \bfX_{\lambda})$ is indeed in $\SSubm{\alpha}{\overline \bfA, \overline \pi}{\uM}$, i.e. it is a $\Gamma$-invariant spacial local sublattice of $\overline \bfA$, containing $\alpha$.

To show that $\widetilde \uM(\coprod\limits_{\lambda \in \Lambda} \bfX_{\lambda})$ is the final object in $\SSubm{\alpha}{\overline \bfA, \overline \pi}{\uM}$, we need to show that for every $\bfZ \in \SSubm{\alpha}{\overline \bfA, \overline \pi}{\uM}$ there is a spacial morphism $\psi: \bfZ \to \widetilde \uM(\coprod\limits_{\lambda \in \Lambda} \bfX_{\lambda})$. By the definition of the category of spacial sublattices, we have $\bfZ  = \widetilde \uM(\bfX_0)$ for some factor $\bfX_0$ of $\bfA$. Then $\bfX_0$ belongs to the collection of subobjects $\{ \bfX_{\lambda}\}_{\lambda \in \Lambda}$ from the statement of the theorem, and so there is a morphism $\psi_0: \bfX_0 \to \coprod\limits_{\lambda \in \Lambda} \bfX_{\lambda}$ in category $\Suba{\bfA,\pi}$. Then $\widetilde \uM(\psi_0): \bfZ \to \widetilde \uM(\coprod\limits_{\lambda \in \Lambda} \bfX_{\lambda})$ is the spacial morphism as required.

In the language of lattices this implies that $\widetilde \uM(\coprod\limits_{\lambda \in \Lambda} \bfX_{\lambda})$ is a spacial sublattice containing $\alpha$, and it is embedded in any other spacial sublattice with this property.
\end{proof}

This proposition has the following corollary, whose proof is straightforward.
\begin{cor}
\label{cor.coprodcond}
Let $\uM: \Rep{\Gamma}{\uC} \to \Rep{\Gamma}{\MDL}$ be a covariant functor satisfying the conditions of Proposition \ref{p.coprodstab}. Then $\uM$ is a measurement functor if and only if for every $(\bfA,\pi) \in \ob(\Rep{\Gamma}{\uC})$ with the corresponding abstract dynamical lattice $(\overline \bfA, \overline \pi)=\uM(\bfA,\pi)$ and the localization function $\Omega$ we have for every cover $\alpha \in \Sigma_{\overline \bfA}$ that $\Omega(\alpha) = \widetilde \uM(\coprod\limits_{\lambda \in \Lambda} \bfX_{\lambda})$, where $\{ \bfX_{\lambda} \}_{\lambda \in \Lambda}$ is the collection of all objects in $\Suba{\bfA,\pi}$ s.t. $\widetilde \uM(\bfX_{\lambda})$ is in $\SSubm{\alpha}{\overline \bfA, \overline \pi}{\uM}$ for all indices $\lambda$.
\end{cor}
We now return to the examples that we considered in Sections \ref{ex.toplatint}, \ref{ex.measlatint} and later in Sections \ref{sss.coptop}, \ref{sss.copmeas}.

\subsubsection{Topological dynamics}
\label{sss.topdynex}
Consider the category $\TOP$ of compact Hausdorff spaces with surjective continuous maps as morphisms and the associated category $\Rep{\Gamma}{\TOP}$, which is just the category of topological dynamical systems. Objects of $\Rep{\Gamma}{\TOP}$ are pairs $(\bfA, \pi)$, where $\bfA$ is a compact topological space and $\pi: \Gamma \to \Homeo(\bfA)$ is a homomorphism from $\Gamma$ to the group of homeomorphisms of $\bfA$. We will describe the construction of a measurement functor $\uM_{\TOP}: \Rep{\Gamma}{\TOP} \to \Rep{\Gamma}{\MDL}$ that eventually leads to the definition of topological entropy.

Given $(\bfA, \pi)$ as above, $\bfA$ is a nonempty topological space $\bfA=(\uK,\calU)$ with a compact topology $\calU$ on a set $\uK$, define $\uM_{\TOP}(\bfA,\pi):=(\calU,m,\Omega; \overline{\pi})$. Here $\calU$ is the distributive lattice of open sets with $0$ being the empty set and $1$ being the whole space $\uK$, $m$ takes value $1$ on all $a \in \calU, a\neq 0$ and $m(0)=0$. A finite cover $\alpha \in \Sigma_{\calU}$ is then simply a finite cover of $\uK$ by open sets. We define the group homomorphism $\overline{\pi}$ by
\begin{equation*}
\overline{\pi}(\gamma): U \mapsto \pi_{\gamma}^{-1} U, \ \ U \in \calU, \gamma \in \Gamma.
\end{equation*}

For $\Omega(\alpha)$ we take the minimal $\Gamma$-invariant topology $\calV \subseteq \calU$ on $\uK$ that contains the family of open sets $\alpha$. Note that this topology is compact and has a basis of open sets
\begin{align*}
\{ \overline \pi_{\gamma_1} a_1 \cap \overline \pi_{\gamma_2} a_2 \cap \dots \cap \overline \pi_{\gamma_n} a_n \},
\end{align*}
where $n$ runs through $\N$, $\gamma_1,\dots,\gamma_n$ run over $\Gamma$, $a_1,\dots,a_n$ run over $ \alpha$. This topology is in general not Hausdorff. We describe the action of $\uM_{\TOP}$ on morphisms. Let $(\bfA,\pi)$  and $(\bfB,\rho)$ in $\Rep{\Gamma}{\TOP}$ be topological dynamical systems with $\bfA=(\uK,\calU), \bfB=(\uL,\calV)$ s.t. $$\uM_{\TOP}(\bfA,\pi)=(\calU,m_A,\Omega_A; \overline{\pi}), \ \ \uM_{\TOP}(\bfB,\rho)=(\calV,m_B,\Omega_B; \overline{\rho}).$$ Let  $\phi: \bfA \to \bfB$ be a $\Gamma$-equivariant morphism, that is a surjective continuous map $\phi: \uK \to \uL$ commuting with the action of $\Gamma$, then we define morphism $\uM_{\TOP}(\phi)$ as follows. Consider distributive lattice embedding $\Phi: \calV \to \calU$, $a \mapsto \phi^{-1} a$ for $a \in \calV$. Then by the surjectivity of $\phi$ set $\Phi(a)$ is nonempty if and only if $a$ is nonempty, thus $m_A(\Phi(a))=m_B(a)$. Let $\alpha \subseteq \calV$ be an open cover, then $\Omega_B(\alpha)$ is minimal $\Gamma$-invariant topology $\calV' \subseteq \calV$ on $\uL$ that contains family of open sets $\alpha$, and we have $\Phi(\Omega_B(\alpha)) = \phi^{-1} \calV' \subseteq \calU$. The topology $\phi^{-1} \calV'$  has basis
\begin{align*}
\{ \Phi(\overline \rho_{\gamma_1} a_1 \cap \overline \rho_{\gamma_2} a_2 \cap \dots \cap \overline \rho_{\gamma_n} a_n): &n \in \N, \gamma_1,\dots,\gamma_n \in \Gamma, \\
&a_1,\dots,a_n \in \alpha \}.
\end{align*}
Similarly, $\Omega_A(\Phi(\alpha))$ is the topology with basis sets of the form $$\overline \pi_{\gamma_1} (\Phi a_1) \cap \overline \pi_{\gamma_2} (\Phi a_2) \cap \dots \cap \overline \pi_{\gamma_n} (\Phi a_n).$$ The map $\phi$ is $\Gamma$-equivariant, so
\begin{align*}
&\overline \pi_{\gamma_1} (\Phi a_1) \cap \overline \pi_{\gamma_2} (\Phi a_2) \cap \dots \cap \overline \pi_{\gamma_n} (\Phi a_n) =\\
&=\Phi(\overline \rho_{\gamma_1} a_1 \cap \overline \rho_{\gamma_2} a_2 \cap \dots \cap \overline \rho_{\gamma_n} a_n)
\end{align*}
and the topologies coincide. This shows that $\Phi^{\op} \in \Hom_{\Gamma}((\bfA,\pi),(\bfB,\rho))$.

The fact that $\uM_{\TOP}(\phi \circ \psi) = \uM_{\TOP}(\phi) \circ \uM_{\TOP}(\psi)$ is also easily verified. Hence $\uM_{\TOP}$ is a well-defined covariant functor.

We explain briefly how one can use Proposition \ref{p.coprodstab} to verify the main condition in the definition of a measurement functor. It relies on the observations in Section \ref{sss.copmeas}, so first of all $\Suba{\bfA,\pi}$ is a small category. Secondly, if one follows all the steps in the construction of the coproduct of two factors, one notes that the resulting subtopology contains a given cover $\alpha$ if both factors do. Thus we can apply Proposition \ref{p.coprodstab} and Corollary \ref{cor.coprodcond}, since $\Omega(\alpha)$ defined above coincides with $\widetilde \uM(\coprod\limits_{\lambda \in \Lambda} \bfX_{\lambda})$ from the proposition.

\subsubsection{Measure-preserving dynamics}
\label{sss.measdynex}
Let us return to the category $\PSP$ of standard probability spaces with equivalence classes of measure-preserving maps as morphisms and the associated category $\Rep{\Gamma}{\PSP}$. In what follows, a measurement functor $\uM_{\PSP}: \Rep{\Gamma}{\PSP} \to \Rep{\Gamma}{\MDL}$, that can be used to define Kolmogorov-Sinai entropy, is constructed. Given probability-preserving dynamical system $(\bfA, \pi)$, where the standard probability space $\bfA=(\uK,\calA,\mu)$ has measure algebra $\calA$ and probability measure $\mu$, define $\uM_{\PSP}(\bfA,\pi):=(\calA,m,\Omega; \overline{\pi})$. Here $\calA$ carries the structure of distributive lattice with $0$ being the empty set and $1$ being the whole space $\uK$, and $m(a):=\mu(a)$ for all $a \in \calA$.
We define group homomorphism $\overline{\pi}$ by $\overline{\pi}(\gamma): U \mapsto \pi_{\gamma}^{-1} U$ for $U \in \calA, \gamma \in \Gamma$.
A finite cover $\alpha \in \Sigma_{\calA}$ is a finite cover of $\uK$ by measurable sets modulo null sets, so for $\Omega_A(\alpha)$ we take the minimal $\Gamma$-invariant measure subalgebra of $\calA$ on $\uK$ that contains the family of measurable sets $\alpha$. Note that this subalgebra equals $\sigma(\calA_0)$, for $\calA_0$ being the algebra of unions of sets of the form
$$\bigcap\limits_{i=1}^n \overline \pi_{\gamma_i} a_i \cap \bigcap\limits_{i=n+1}^{n+m} \overline \pi_{\gamma_i} a_i^c,$$ where $n, m \in \N, \gamma_1,\dots,\gamma_{n+m} \in \Gamma, a_1,\dots,a_{n+m} \in \alpha$.

Let $(\bfA,\pi), (\bfB,\rho) \in \Rep{\Gamma}{\PSP}$ be probabilistic dynamical systems s.t. $$\uM_{\PSP}(\bfA,\pi)=(\calA,m_A,\Omega_A; \overline{\pi}), \ \  \uM_{\PSP}(\bfB,\rho)=(\calB,m_B,\Omega_B; \overline{\rho}).$$ Let  $\phi: \bfA \to \bfB$ be a $\Gamma$-equivariant morphism, then we define the morphism $\uM_{\PSP}(\phi)$ as follows. Consider the measure algebra embedding $\Phi: \calB \to \calA$, $a \mapsto \phi^{-1} a$ for $a \in \calB$. Then $m_A(\Phi(a))=m_B(a)$, since $\phi$ is measure-preserving. For a finite cover $\alpha \in \Sigma_{\calB}$ the sigma-algebra $\Phi \Omega_B(\alpha)$ equals $\Phi \sigma(\calB_0)$. Here $\calB_0$ is the algebra of unions of sets of the form $\bigcap\limits_{i=1}^n \overline \pi_{\gamma_i} a_i \cap \bigcap\limits_{i=n+1}^{n+m} \overline \pi_{\gamma_i} a_i^c$, where $n, m \in \N, \gamma_1,\dots,\gamma_{n+m} \in \Gamma, a_1,\dots,a_{n+m} \in \alpha$. But $\sigma(\Phi \calB_0) = \Phi \sigma(\calB_0)$; so $\Phi \sigma(\calB_0)$ coincides with the sigma-algebra generated by the algebra $\Phi \calB_0$, which is just $\Omega_A(\Phi \alpha)$.

It is also easy to see that the functor $\uM_{\PSP}$ respects composition of morphisms, thus it is a covariant functor. The main condition in the definition of a measurement functor follows from a similar application of Proposition \ref{p.coprodstab} and Corollary \ref{cor.coprodcond}, which we can apply because of the stability of the coproduct construction in Section \ref{sss.copmeas}.

\section{Entropy of representations on abstract categories}
\label{s.entropy}

\subsection{Preliminaries on amenable groups} \label{ss.amenprelim}
A sequence of finite subsets $( F_n )_{n \in \N}$ of $\Gamma$ is called a \textbf{F{\o}lner sequence} if for every $g \in \Gamma$
\begin{equation}
\label{eq.folnerseq}
\lim\limits_{n \to \infty} \frac {|F_n \cap g F_n|} {|F_n|} = 1.
\end{equation}

It is easy to see that a sequence $( F_n )_n$ of finite subsets of $\Gamma$ is a F{\o}lner sequence if and only if for every $g \in \Gamma$ one has $$\lim\limits_{n \to \infty} \frac {|F_n \sdif g F_n|} {|F_n|} = 0.$$ A countable, infinite group $\Gamma$ is called \textbf{amenable} if it admits a F{\o}lner sequence.

In the definition of entropy we shall use the following important lemma, whose proof can be found in \cite{gromovo} or \cite{krieger}. We remind the reader that $\FPow(\Gamma)$ denotes the set of all finite subsets of the group $\Gamma$.
\begin{prop}[Ornstein-Weiss lemma]
\label{p.ornsteinweiss}
Let $f: \FPow(\Gamma) \to \Rp$ be a function from the set of finite subsets of an amenable group $\Gamma$ to the set of non-negative reals satisfying the following conditions

\begin{aufzi}
\item $f$ is monotone, i.e. $f(F_1) \leq f(F_2)$ holds for any two finite subsets $F_1 \subseteq F_2 \subsetneq \Gamma$;

\item $f$ is subadditive, i.e. $f(F_1 \cup F_2) \leq f(F_1) + f(F_2)$ holds for any two finite subsets $F_1,F_2 \subsetneq \Gamma$;

\item $f$ is right-invariant, i.e. $f(F g) = f(F)$ holds for all finite $F \subsetneq \Gamma$ and $g \in \Gamma$.
\end{aufzi}
Then for every F{\o}lner sequence $( F_n )_{n \in \N}$ of $\Gamma$ the limit
\begin{equation}
\label{eq.ornsteinweiss}
\lim\limits_{n\to \infty} \frac {f(F_n)} {|F_n|}
\end{equation}
in $\Rp$ exists and is independent of the choice of F{\o}lner sequence.
\end{prop}

\subsection{Entropy of abstract dynamical lattices} \label{ss.entdynlat}
In this subsection we let $$\bfV=(V,m,\Omega; \pi)$$ be an arbitrary abstract dynamical lattice in $\Rep{\Gamma}{\MDL}$ with a representation $\pi$ of a discrete amenable group $\Gamma$. Since morphisms in $\MDL$ are defined as opposites of the corresponding lattice embeddings (Section \ref{ss.dynlatdef}), the representation $\pi$ of $\Gamma$ determines (canonically) a left action of $\Gamma^{\op}$ on $V$ by lattice embeddings. Also, the representation $\pi$ is fixed throughout the largest part of this subsection, hence we suppress $\pi$ in the notation when possible and write $g x, g \in \Gamma, x \in V$ for the action of $\Gamma^{\op}$. In particular, we have the identity $(fg)x = g(fx)$ for all $f,g \in \Gamma$ and $x \in V$. For an arbitrary $\Gamma$-invariant sublattice  $W$ of $V$ we have $\Sigma_W \subseteq \Sigma_V$.

Since $V$ is a distributive lattice, it is equipped with a partial ordering relation $\leq$. To describe how covers are related to each other we need an ordering of covers as well. We cannot get a useful partial ordering relation on the set of covers in general, but there is a `natural' relation that is \emph{not} antisymmetric. Namely, the set $\Sigma_V$ is equipped with a quasiorder relation
\begin{equation*}
\succeq: \beta \succeq \alpha \ \Longleftrightarrow \ \forall \ b \in \beta \ \ \exists \ a \in \alpha \ \text{ s.t.  } a \geq b.
\end{equation*}
This coincides with the definition from the theory of topological entropy, where an open cover $\calU$ is said to be finer than an open cover $\calV$ if for every open set $A \in \calU$ there is an open set $B \in \calV$ s.t. $A \subseteq B$. If $\alpha \succeq \beta$, we say that $\alpha$ is \textbf{finer} than $\beta$, or that $\alpha$ \textbf{refines} $\beta$.

Given covers $\alpha, \beta \in \Sigma_W$ we define the \textbf{join} of these covers by
\begin{equation*}
\alpha \Sup \beta:=\{ a \Inf b: a\in \alpha, b \in \beta \} \in \Sigma_W.
\end{equation*}
The binary operation $\Sup$ is associative and commutative, hence we can also talk about joins of finite sets of covers. Similar to the remark made above, we cannot call a join of covers a supremum, but it is easy to see that $\gamma \succeq (\alpha \Sup \beta)$ if and only if $\gamma \succeq \alpha$ and $\gamma \succeq \beta$\footnote{It follows from this remark that $\alpha \Sup \beta$ is in fact a product of $\alpha$ and $\beta$ in a suitably defined `quasi-order category' $(\Sigma_V, \succeq)$.}. Furthermore,
\begin{equation}
\label{eq.covvmonot}
\alpha \Sup \beta \succeq \alpha \ \ \forall \alpha, \beta \in \Sigma_V.
\end{equation}

For an element $f \in \Gamma$ and a cover $\alpha \in \Sigma_W$ define $f \alpha:=\{ f x: x \in \alpha\} \in \Sigma_W$. This yields an action of  $\Gamma^{\op}$ on the set of covers $\Sigma_W$, hence for all $f,g \in \Gamma$ we have $(fg) \alpha = g(f \alpha)$. Given a finite subset $F \subset \Gamma$ define $\alpha^F:= \bigvee_{f \in F} f \alpha \in \Sigma_W$. It is clear that the action of $\Gamma^{\op}$ on $\Sigma_W$ preserves the preorder relation $\succeq$. Furthermore, $\beta \succeq g^{-1} \alpha$ if and only if $g \beta \succeq \alpha$. Also, for all covers $\alpha, \beta \in \Sigma_W$ and $g \in \Gamma$
\begin{align*}
g \alpha \Sup g \beta &= \{ g a \Inf g b: a \in \alpha, b \in \beta\} = \\
&=g \{ a \Inf b: a \in \alpha, b \in \beta\} = g(\alpha \Sup \beta).
\end{align*}

Let $\alpha \in \Sigma_V$, and let $W \in \SLPowg{\alpha}{V}$ be a $\Gamma$-invariant sublattice containing $\alpha$. Our goal now is to define the entropy of a cover $\alpha$ relative to the sublattice $W$. Define the \textbf{total mass} of a cover $\alpha$ by
\begin{equation*}
S(\alpha):=\sum\limits_{a \in \alpha} m(a).
\end{equation*}
Since $\alpha$ is a cover, there exists an element $a \in \alpha$ with $m(a)>0$, hence $S(\alpha)$ is always strictly positive. Define
\begin{equation}
\label{eq.hstar}
h^*(\alpha):=-\sum\limits_{a \in \alpha} \frac {m(a)} {S(\alpha)} \log \frac {m(a)} {S(\alpha)}.
\end{equation}
By a standard convention, we assume that $0 \cdot \log 0 = 0$. Then $h^*(\alpha)$ is always a nonnegative real number.

For a cover $\alpha \in \Sigma_V$ we denote the number of nonzero elements of $\alpha$ by
\begin{equation*}
N(\alpha):=\left| \{ a \in \alpha: m(a) \neq 0\} \right|.
\end{equation*}
Then $N(\alpha) \geq 1$ and $N(g \alpha) = N(\alpha)$ for every $g \in \Gamma$. Furthermore, $h^*(\alpha) \leq \log N(\alpha)$. Given a cover $\alpha \in \Sigma_V$ and an invariant sublattice $W \in \SLPowg{\alpha}{V}$ as above, we define
\begin{equation}
\label{eq.hhat}
\hat{h}_{W}(\alpha):= \sup\{ h^*(\beta): \beta \in \Sigma_{W} \text{ s.t. } \beta \succeq \alpha, N(\beta) \leq N(\alpha)
 \}.
\end{equation}
Of course, $$h^*(\alpha) \leq \hat{h}_{W}(\alpha) \leq \log N(\alpha).$$

Finally, the function $h_{\cdot}(\cdot)$ with values in $\Rp$ can be introduced. It will be used together with the Ornstein-Weiss Lemma to define the entropy of a dynamical lattice $((V,m,\Omega),\pi)$ in Proposition \ref{p.entropdef}. Given a cover $\alpha \in \Sigma_V$ and an invariant sublattice $W \in \SLPowg{\alpha}{V}$ as above, let
\begin{equation}
h_{W}(\alpha):=\inf\{ \sum\limits_{j=1}^n \hat h_W(\beta_j): n \in \N, \ \beta_1, \dots, \beta_n \in \Sigma_W \text{ s.t. } \bigvee\limits_{j=1}^n \beta_j \succeq \alpha\}.
\end{equation}
Since $\alpha \in \Sigma_W$, $$0 \leq h_{W}(\alpha) \leq \hat{h}_{W}(\alpha) \leq \log N(\alpha),$$ thus $h_{W}(\alpha)$ is always a nonnegative real number. Hence we can define a function $f_{\alpha,W}: \FPow(\Gamma) \to \Rp$ by $f_{\alpha,W}(F):=h_W(\alpha^F)$. We are now able to prove the main proposition.
\begin{prop}
\label{p.satornweiss}
For each cover $\alpha \in \Sigma_V$ and each invariant sublattice $W \in \SLPowg{\alpha}{V}$ the function $f_{\alpha,W}$ satisfies the conditions of Ornstein-Weiss lemma (Proposition \ref{p.ornsteinweiss}).
\end{prop}
\begin{proof}
(i) Let us show that $f_{\alpha,W}$ is monotone, i.e. for two arbitrary finite subsets $F_1 \subseteq F_2 \subset \Gamma$ implies that $f_{\alpha,W}(F_1) \leq f_{\alpha,W}(F_2)$. Indeed, let $F_3 = F_2 \setminus F_1$, then $\alpha^{F_2} = \alpha^{F_1} \Sup \alpha^{F_3} \succeq \alpha^{F_1}$ by equation (\ref{eq.covvmonot}). Thus

\begin{align*}
&h_W(\alpha^{F_1})=\inf\{ \sum\limits_{j=1}^n \hat h_W(\beta_j): n \in \N, \ \beta_1, \dots, \beta_n \in \Sigma_W \text{ s.t. } \bigvee\limits_{j=1}^n \beta_j \succeq \alpha^{F_1} \}\\
&\leq \inf\{ \sum\limits_{j=1}^n \hat h_W(\beta_j): n \in \N, \ \beta_1, \dots, \beta_n \in \Sigma_W \text{ s.t. } \bigvee\limits_{j=1}^n \beta_j \succeq \alpha^{F_1} \Sup \alpha^{F_3} \} \\
&= h_W(\alpha^{F_2}).
\end{align*}

(ii) Now we show that $f_{\alpha,W}$ is subadditive, i.e. $f_{\alpha,W}(F_1 \cup F_2) \leq f_{\alpha,W}(F_1) + f_{\alpha,W}(F_2)$ holds for two arbitrary finite subsets $F_1,F_2 \subset \Gamma$. For that observe that $\alpha^{F_1} \Sup \alpha^{F_2} \succeq \alpha^{F_1 \cup F_2}$, hence $h_W(\alpha^{F_1 \cup F_2}) \leq h_W(\alpha^{F_1} \Sup \alpha^{F_2})$ by a monotonicity argument. Thus it suffices to show that for any two $\alpha,\beta \in \Sigma_W$ the inequality $h_W(\alpha \Sup \beta) \leq h_W(\alpha) + h_W(\beta)$ holds.

Indeed, for any $k,l \in \N$ and sequences of covers $( \beta_{i_2}' )_{i_2=1}^k, ( \beta_{i_3}'' )_{i_3=1}^l$ in $\Sigma_W$ s.t. $\bigvee\limits_{i_2=1}^k \beta_{i_2}' \succeq \alpha$ and $\bigvee\limits_{i_3=1}^l \beta_{i_3}'' \succeq \beta$ the sequence of covers $\beta_1',\beta_2',\dots,\beta_k', \beta_1'',\beta_2'',\dots,\beta_l''$ is also in $\Sigma_W$ and its join refines $\alpha \Sup \beta$. Hence
\begin{align*}
&h_W(\alpha \Sup \beta) \leq h_W(\alpha)+h_W(\beta).
\end{align*}

(iii) Finally, we prove that $f_{\alpha,W}$ is right-invariant, i.e. $f_{\alpha,W}(F g) = f_{\alpha,W}(F)$ holds for every finite $F \subsetneq \Gamma$ and every element $g \in \Gamma$. Observe that $$\bigvee\limits_{f \in F g} f \alpha = \bigvee\limits_{f \in F} (fg) \alpha = g \left(\bigvee\limits_{f \in F} f \alpha \right).$$ Hence it suffices to show that $h_W (g \alpha) = h_W (\alpha)$ holds for arbitrary $\alpha \in \Sigma_V, g \in \Gamma$ and $W$ being arbitrary $\Gamma$-invariant sublattice containing $\alpha$.

Indeed,
\begin{align*}
& h_W(g \alpha) = \inf\{ \sum\limits_{j=1}^n \hat h_W(\beta_j): n \in \N,\ \beta_1,\dots,\beta_n \in \Sigma_W \text{ s.t. } \bigvee\limits_{j=1}^n \beta_j \succeq g \alpha\} \\
&=\inf\{ \sum\limits_{j=1}^n \hat h_W(g(g^{-1} \beta_j)):  n \in \N,\ \beta_1,\dots,\beta_n \in \Sigma_W \text{ s.t. } \bigvee\limits_{j=1}^n (g^{-1} \beta_j) \succeq \alpha \}
\end{align*}
hence it suffices to show that $\hat h_W(g \alpha) = \hat h_W(\alpha)$ for every $\alpha \in \Sigma_V, g \in \Gamma$ and $W$ arbitrary $\Gamma$-invariant sublattice containing $\alpha$. Indeed,
\begin{align*}
\hat h_W(g \alpha ) &= \sup\{ h^*(\beta): \beta \in \Sigma_{W} \text{ s.t. } \beta \succeq g \alpha, N(\beta) \leq N( g \alpha) \} = \\
&=\sup\{ h^*(g(g^{-1} \beta)): g^{-1} \beta \in \Sigma_{W} \text{ s.t. } g^{-1} \beta \succeq \alpha, N(\beta) \leq N(\alpha) \},
\end{align*}
thus we only need to show that $h^*(g \alpha) = h^*(\alpha)$ for any $\alpha \in \Sigma_V, g \in \Gamma$. This in turn follows from the definition of morphism in category $\MDL$, more precisely we use that it preserves measurement function.

\end{proof}

Now we are ready to define the dynamical entropy of a cover $\alpha \in \Sigma_V$ relative to a sublattice $W$ containing $\alpha$.
\begin{prop}[Entropy of a cover relative to sublattice]
\label{p.entropdef}
Consider an abstract dynamical lattice $(V,m,\Omega; \pi)$ with a representation $\pi$ of a discrete amenable group $\Gamma$, and let $(F_n )_{n \in \N}$ be a F{\o}lner sequence in $\Gamma$. Then for all $\alpha \in \Sigma_V$ and all invariant sublattices $W \in \SLPowg{\alpha}{V}$ the limit
\begin{equation}
\label{eq.entropdef}
\lim\limits_{n \to \infty} \frac {h_W(\alpha^{F_n})} {|F_n|}=:h_W(\alpha,\pi)
\end{equation}
exists, is nonnegative and is independent of the F{\o}lner sequence $( F_n )_n$.
\end{prop}
\begin{proof}
Follows from Proposition \ref{p.satornweiss} and Proposition \ref{p.ornsteinweiss}.
\end{proof}

\label{def.entdynlat}
Now we mimic the definitions of both Kolmogorov-Sinai and topological entropies and define the entropy $h_{\MDL}$ of a dynamical lattice $(V,m,\Omega; \pi)$ by
\begin{equation*}
h_{\MDL}(V,m,\Omega; \pi):=\sup\{ h_{\Omega(\alpha)}(\alpha,\pi): \alpha \in \Sigma_V \}.
\end{equation*}

This notion of entropy enjoys a very useful monotonicity property that we will use in the next section to prove some of the key results of this article.
\begin{prop}
\label{p.entmonot}
Let $\bfA = (W,m_W,\Omega_W)$ and $\bfB=(V,m_V,\Omega_V)$ be measured lattices with localization with representations $\pi,\rho$ of $\Gamma$. Let $(\bfA,\pi) \overset{\psi}{\to} (\bfB,\rho)$ be a morphism in $\Hom_{\Gamma}(\bfA,\bfB)$. Then $h_{\MDL}(\bfA,\pi) \geq h_{\MDL}(\bfB,\rho)$.
\end{prop}
\begin{proof}
Let $\psi = \Phi^{\op}$, where $\Phi$ is the corresponding embedding. Then
\begin{align*}
&h_{\MDL}(W,m_W,\Omega_W; \pi)= \\
&=\sup\{ h_{\Omega_W(\alpha)}(\alpha, \pi): \alpha \in \Sigma_W \} \geq \sup\{ h_{\Omega_W(\Phi(\alpha))}(\Phi(\alpha), \pi): \alpha \in \Sigma_V \} = \\
& \overset{(\ast)}{=} \sup\{ h_{\Phi(\Omega_V(\alpha))}(\Phi(\alpha), \pi): \alpha \in \Sigma_V \} \overset{(\ast \ast)}{=} \sup\{ h_{\Omega_V(\alpha)}(\alpha, \rho): \alpha \in \Sigma_V\} = \\
&= h_{\MDL}(V,m_V,\Omega_V; \rho).
\end{align*}
Apart from the equality $(\ast \ast)$, the proof is rather straightforward. In the equality $(\ast)$ we use that $\Omega_W(\Phi(\alpha)) = \Phi(\Omega_V(\alpha))$ by the definition of morphisms in $\MDL$. In the proof of $(\ast \ast)$ we have used that $m_W(\Phi(a))=m_V(a)$ for all $a \in \alpha \in \Sigma_V$, and that the morphism $\Phi$ is a lattice homomorphism intertwining the action of $\Gamma$. Hence the entropy of $\alpha$ computed with respect to $\Omega(\alpha)$ in the lattice $V$ equals the entropy of $\Phi(\alpha)$ computed with respect to $\Phi(\Omega(\alpha))$ in the lattice $W$.

\end{proof}

This result admits a categorical interpretation.
\begin{cor}\label{cor.entfunctor}
The correspondence $h_{\MDL}: (\bfA,\pi) \mapsto h_{\MDL}(\bfA,\pi) \in [0,\infty]$ is a covariant functor from $\Rep{\Gamma}{\MDL}$ to the poset category $[0,\infty]$ of extended positive reals.
\end{cor}
\begin{proof}
Objects of $[0,\infty]$ are extented positive reals. For $x,y \in [0,\infty]$ the set $\Hom(x, y)$ consists of exactly one arrow $\geq$ if and only if $x \geq y$, and is empty otherwise. Then the statement follows from the previous proposition and the transitivity of relation.
\end{proof}

\subsection{Entropy of representations on abstract categories}
Let $\uC$ be a category, $\Rep{\Gamma}{\uC}$ be the associated category of representations of a discrete amenable group $\Gamma$, and $\uM: \Rep{\Gamma}{\uC} \to \Rep{\Gamma}{\MDL}$ be a measurement functor. We define the entropy of the representation $(\bfA,\pi) \in \Rep{\Gamma}{\uC}$ associated with the measurement functor $\uM$ by
\begin{equation}
h((\bfA,\pi),\uM):=h_{\MDL}(\uM(\bfA,\pi)).
\end{equation}

An important property of topological and Kolmogorov-Sinai entropy is that it decreases when passing to factors. One of the main results of this article is that our abstractly defined entropy also does decrease when moving down the arrows.
\begin{prop}[Left Entropic Inequality]
\label{p.leftentineq}
Let $(\bfA,\pi) \overset{\psi}{\to} (\bfB,\rho)$ be a morphism in $\Hom_{\Gamma}(\bfA,\bfB)$. Then $h((\bfA,\pi), \uM) \geq h((\bfB,\rho), \uM)$.
\end{prop}
\begin{proof}
Suppose $\uM(\bfA,\pi) = (\overline \bfA,\overline \pi)$, $\uM(\bfB,\rho)=(\overline \bfB,\overline \rho)$. Here $\overline \bfA = (W,m_W,\Omega_W)$, $\overline \bfB=(V,m_V,\Omega_V)$ are measured lattices with localization and $\uM(\psi) = \Phi^{\op}: \overline \bfA \to \overline \bfB$ is the image of morphism $\psi$. Then the statement follows immediately from Proposition \ref{p.entmonot}.
\end{proof}

In the notation for the entropy of $(\bfA,\pi) \in \Rep{\Gamma}{\uC}$ we always say explicitly what measurement functor we use. It is natural to ask how different measurement functors can be compared, i.e., what are the `arrows' between measurement functors. If there is a morphism between two different measurement functors, do we also get monotonicity of corresponding entropies?

To answer these questions we recall the standard notion of a \textbf{natural transformation} between functors first. So let for the moment $\uC,\uD$ be two arbitrary categories and $\calF, \calG: \uC \to \uD$ be covariant functors. A family of morphisms $\alpha = \{ \alpha_{\bfX} \}_{\bfX \in \ob(\uC)}$ in $\uD$, where $\alpha_{\bfX}: \calF(\bfX) \to \calG(\bfX)$ for every $\bfX \in \ob(\uC)$, is called a natural transformation between functors $\calF$ and $\calG$ if for every morphism $\phi: \bfX \to \bfY$ in $\uC$ the diagram
\begin{equation*}
 \xymatrix{ \calF(\bfX) \ar[d]_{\alpha_{\bfX}} \ar[r]^{\calF(\phi)} & \calF(\bfY) \ar[d]^{\alpha_{\bfY}} \\
               \calG(\bfX) \ar[r]_{\calG(\phi)} & \calG(\bfY)  }
\end{equation*}
commutes. Then we write $\alpha: \calF \to \calG$, i.e. we call $\alpha$ an arrow between functors. This can be justified by defining the \textbf{functor category} $[\uC,\uD]$ whose objects are covariant functors from $\uC$ to $\uD$ and whose arrows are natural transformations\footnote{There is a subtlety here: the collection of all natural transformations between functors is not necessarily a small set, so the `category of functors' is not necessarily locally small.}. Given functors $\calF, \calG, \calH$ with natural transofmations $\alpha: \calF \to \calG$ and $\beta: \calG \to \calH$, we see that $\gamma = \{ \beta_{\bfX} \alpha_{\bfX}\}_{\bfX \in \ob(\uC)}$ is a natural transformation between $\calF$ and $\calH$. We call a natural transformation $\alpha: \calF \to \calG$ a \textbf{natural equivalence} if $\alpha_{\bfX}$ is an isomorphism for every $\bfX \in \ob(\uC)$.

We return back to the main topic, so now $\uC$ is some category and $\Rep{\Gamma}{\uC}$ is the associated category of representations of a discrete amenable group $\Gamma$. Then the collection of all measurement functors from $\Rep{\Gamma}{\uC}$ to $\Rep{\Gamma}{\MDL}$ is a full subcategory of $[\Rep{\Gamma}{\uC}, \Rep{\Gamma}{\MDL}]$. We denote the category of measurement functors from $\Rep{\Gamma}{\uC}$ to $\Rep{\Gamma}{\MDL}$ by $[\Rep{\Gamma}{\uC}, \Rep{\Gamma}{\MDL}]_{\frakM}$. Observe that for a nontrivial\footnote{I.e. nonempty.} category $\uC$ the categories $\Rep{\Gamma}{\uC}$ and $[\Rep{\Gamma}{\uC}, \Rep{\Gamma}{\MDL}]_{\frakM}$ are also nontrivial. Indeed, there is always a trivial measurement functor $\uM_0: \Rep{\Gamma}{\uC} \to \Rep{\Gamma}{\MDL}$ that maps a representation $(\bfA,\pi)$ of $\Gamma$ to a trivial representation $(\bfV_0,\mathbf{id})$ of $\Gamma$. Here $\bfV_0 = (V,m,\Omega)$ is a measured lattice with localization with $V=\{ 0, 1\}$, $m(1) = 1, m(0)=0$ and $\Omega(\{ 0, 1\})=V$. Entropy measured with respect to the functor $\uM_0$ is identically zero on $\Rep{\Gamma}{\uC}$.

We now prove the second key `monotonicity' result.
\begin{prop}[Right Entropic Inequality]
\label{p.rightentineq}
Let $\alpha: \uM \to \uN$ be a natural transformation between measurement functors $\uM,\uN: \Rep{\Gamma}{\uC} \to \Rep{\Gamma}{\MDL}$. Then for every object $(\bfA,\pi) \in \Rep{\Gamma}{\uC}$ we have $h((\bfA,\pi),\uM) \geq h((\bfA,\pi),\uN)$. If $\alpha$ is a natural equivalence, then $h((\bfA,\pi),\uM) = h((\bfA,\pi),\uN)$.
\end{prop}
\begin{proof}
Let $\uM(\bfA,\pi) = (\overline \bfA,\overline \pi)$, $\uN(\bfA,\pi)=(\overline \bfB,\overline \rho)$ where $\overline \bfA = (W,m_W,\Omega_W)$, $\overline \bfB=(V,m_V,\Omega_V)$ are measured lattices with localization endowed with representations $\overline \pi,\overline \rho$ of $\Gamma$ respectively.

Then by definition $h((\bfA,\pi),\uM) = h_{\MDL}(\overline \bfA,\overline \pi)$ and $h((\bfA,\pi),\uN) = h_{\MDL}(\overline \bfB,\overline \rho)$. Since $\alpha$ is a natural transformation, we get a morphism $\alpha_{(\bfA,\pi)}: (\overline \bfA,\overline \pi) \to (\overline \bfB,\overline \rho)$ in the category $\Rep{\Gamma}{\MDL}$. Then the statement follows from Proposition \ref{p.entmonot}.
\end{proof}
Of course, we have not used the commutative diagram from the definition of a natural transformation in this proof\footnote{I.e. it would suffice to have an infranatural transformation of measurement functors.}, but we will have to use it in order to derive much stronger results, giving a better functorial interpretation of entropy, below.

Given categories $\Rep{\Gamma}{\uC}$ and $[\Rep{\Gamma}{\uC}, \Rep{\Gamma}{\MDL}]_{\frakM}$, we can define a \textbf{product category} $\Rep{\Gamma}{\uC} \times [\Rep{\Gamma}{\uC}, \Rep{\Gamma}{\MDL}]_{\frakM}$ as follows. Objects of $\Rep{\Gamma}{\uC} \times [\Rep{\Gamma}{\uC}, \Rep{\Gamma}{\MDL}]_{\frakM}$ are all pairs $(\bfA, \uM)$, where $\bfA$ is in $\Rep{\Gamma}{\uC}$ and $\uM$ is a measurement functor in $[\Rep{\Gamma}{\uC}, \Rep{\Gamma}{\MDL}]_{\frakM}$. The collection of morphisms between objects $(\bfA, \uM)$ and $(\bfB, \uN)$ is given by the collection of all pairs $(\phi, \alpha)$, where $\phi: \bfA \to \bfB$ and $\alpha: \uM \to \uN$. Given morphisms $(\phi,\alpha): (\bfA, \uM) \to (\bfB, \uN)$ and $(\psi,\beta): (\bfB, \uN) \to (\bfC, \uL)$, we define their composition componentwise by $(\psi \phi, \beta \alpha): (\bfA, \uM) \to (\bfC, \uL)$.

Given the category $\Rep{\Gamma}{\uC} \times [\Rep{\Gamma}{\uC}, \Rep{\Gamma}{\MDL}]_{\frakM}$, one naturally obtains an \textbf{evaluation bifunctor} $\ev$. On objects, it is given as
\begin{align*}
&(\bfA, \uM) \overset{\ev}{\longmapsto} \uM(\bfA),
\end{align*}
where $(\bfA,\uM)$ is in $\Rep{\Gamma}{\uC} \times [\Rep{\Gamma}{\uC}, \Rep{\Gamma}{\MDL}]_{\frakM}$. For objects $(\bfA,\uM), (\bfB,\uN)$ it works on the corresponding morphisms by
\begin{align*}
(\phi,\alpha) \overset{\ev}{\longmapsto} \alpha_{\bfB} \uM(\phi),
\end{align*}
where $(\phi,\alpha)$ is a morphism in $\Hom((\bfA,\uM),(\bfB,\uN))$.

We claim that this is indeed a covariant functor, i.e. it respects composition of morphisms. That is, if $(\phi,\alpha): (\bfA,\uM) \to (\bfB,\uN)$ and $(\psi,\beta): (\bfC,\uN) \to (\bfC,\uL)$, we want that
\begin{equation}
\label{eq.evbifunc}
\ev((\psi \phi), (\beta \alpha)) \eqdef \alpha_{\bfC} \beta_{\bfC} \uM(\psi \phi)= \beta_{\bfC} \uN(\psi) \alpha_{\bfB} \uM(\phi).
\end{equation}

Applying the definition of the natural transformation $\alpha$ to the morphism $\phi$, we conclude that the diagram
\begin{equation*}
 \xymatrix{ \uM(\bfA) \ar[d]_{\alpha_{\bfA}} \ar[r]^{\uM(\phi)} & \uM(\bfB) \ar[d]^{\alpha_{\bfB}} \\
               \uN(\bfA) \ar[r]_{\uN(\phi)} & \uN(\bfB)  }
\end{equation*}
commutes, hence we get identity $\alpha_{\bfB} \uM(\phi) = \uN(\phi) \alpha_{\bfA}$. Applying the definition of the natural transformation $\alpha$ to the morphism $\psi \phi$, we get that the diagram
\begin{equation*}
 \xymatrix{ \uM(\bfA) \ar[d]_{\alpha_{\bfA}} \ar[r]^{\uM(\psi \phi)} & \uM(\bfC) \ar[d]^{\alpha_{\bfC}} \\
               \uN(\bfA) \ar[r]_{\uN(\psi \phi)} & \uN(\bfC)  }
\end{equation*}
commutes, hence we get identity $\uN(\psi \phi) \alpha_{\bfA} = \alpha_{\bfC} \uM(\psi \phi)$. Substituting these identities in the equation (\ref{eq.evbifunc}), we get
\begin{align*}
&\beta_{\bfC} \uN(\psi) \alpha_{\bfB} \uM(\phi) = \beta_{\bfC} \uN(\psi) \uN(\phi) \alpha_{\bfA} = \beta_{\bfC} \uN(\psi \phi) \alpha_{\bfA} =\beta_{\bfC} \alpha_{\bfC} \uM(\psi \phi),
\end{align*}
which shows that $\ev$ is indeed a covariant functor.

Combining these results, we arrive at
\begin{cor}\label{cor.entropbifunc}
Entropy is a (bi)functor:
\begin{equation*}
\Rep{\Gamma}{\uC} \times [\Rep{\Gamma}{\uC},\Rep{\Gamma}{\MDL}]_{\frakM} \overset{h(\cdot,\cdot)}{\longrightarrow} [0, \infty].
\end{equation*}
\end{cor}
\begin{proof}
This follows immediately, since $((\bfA,\pi),\uM) \mapsto h((\bfA,\pi),\uM)$ is a composition of the evaluation bifunctor $\ev: \Rep{\Gamma}{\uC} \times [\Rep{\Gamma}{\uC}, \Rep{\Gamma}{\MDL}]_{\frakM} \to \Rep{\Gamma}{\MDL}$ and $h_{\MDL}: \Rep{\Gamma}{\MDL} \to [0,\infty]$, which was show to be a functor in Corollary \ref{cor.entfunctor}.
\end{proof}

It is interesting to observe that one can derive the corollary above in a more general - though also completely informal - setting. Consider some category of physical systems $\PSYS$, some category of observables $\OBS$ and the associated category of measurement functors $[\PSYS, \OBS]_{\frakM}$, which is just some full subcategory of $[\PSYS,\OBS]$. Then evaluation $\ev: (\bfA, \uM) \mapsto \uM(\bfA)$ is \emph{still} a bifunctor from $\PSYS \times [\PSYS,\OBS]_{\frakM}$ to $\OBS$. Suppose furthermore that we are given a poset of complexity values $\CVAL$ and a functor $\tilde h: \OBS \to \CVAL$. Then the `complexity' defined via $(\bfA,\uM) \mapsto \tilde h(\uM(\bfA))$ is a (bi)functor from $\PSYS \times [\PSYS,\OBS]_{\frakM}$ to $\CVAL$.

\subsection{Comparison with the classical notions of entropy} \label{ss.edlvscl}
In this subsection we intend to compare the entropies defined via the measurement functors $\uM_{\PSP}$ and $\uM_{\TOP}$ (introduced in Section \ref{ss.measfunctd}) with the classical notions of Kolmogorov-Sinai and topological entropy for amenable group actions.

\begin{prop}
\label{p.compksent}
Let $\bfX = (\prX, \calM, \mu)$ be a standard probability space with measure algebra $\calM$, $(\bfX,\pi) \in \ob(\Rep{\Gamma}{\PSP})$ be a measure-preserving dynamical system, equipped with action of an amenable group $\Gamma$. Denote by $h_{\KS}(\bfX,\pi)$ the Kolmogorov-Sinai entropy of the system $(\bfX,\pi)$. Then
\begin{equation*}
h_{\KS}(\bfX,\pi) = h((\bfX,\pi),\uM_{\PSP}).
\end{equation*}
\end{prop}
\begin{proof}
Recall that the classical entropy of a finite partition $\alpha$ of $\prX$ with respect to a probability measure $\mu$ is defined as
\begin{equation}
\label{eq.entroppart}
h_{\mu}(\alpha,\pi) = \lim\limits_{n \to \infty} \frac 1 {|F_n|} h_{\mu}(\alpha^{F_n}),
\end{equation}
where $h_{\mu}(\beta) = - \sum\limits_{b \in \beta} \mu(b) \log{\mu(b)}$ and $( F_n )_n$ is a F{\o}lner sequence. Then the classical Kolmogorov-Sinai entropy of system $(\bfX,\pi)$ is given by
\begin{equation}
\label{eq.ksentrop}
h_{\KS}(\bfX,\pi) = \sup\{ h_{\mu}(\alpha,\pi): \alpha \text{ finite partition of } \prX \}.
\end{equation}

Let $(\overline \bfX,\overline \pi) = \uM_{\PSP}(\bfX,\pi)$ be the abstract dynamical lattice associated to the system $(\bfX,\pi)$, where $\overline \bfX = (\calM,\mu,\Omega)$ is the measured lattice with localization. We view the measure algebra $\calM$ as a distributive lattice. Measure $\mu$ plays the role of a measurement function, and for a cover $\alpha \in \Sigma_{\calM}$ the localization $\Omega(\alpha)$ is the smallest $\Gamma$-invariant measure subalgebra of $\calM$ containing the family of sets $\alpha$.

Observe that if $\alpha \in \Sigma_{\calM}$ is a cover, then there exists a partition $\alpha' \in \Sigma_{\Omega(\alpha)}$ with $\alpha' \succeq \alpha$ and $\Omega(\alpha')=\Omega(\alpha)$. We call any such $\alpha'$ a \emph{generating disjoint refinement} of $\alpha$. It is essential in this observation that $\Omega(\alpha)$ is a subalgebra containing $\alpha$, it allows to find the required refinement without leaving the sublattice $\Omega(\alpha)$. So, since $\alpha' \succeq \alpha$ and $\Omega(\alpha')=\Omega(\alpha)$, the inequality $h_{\Omega(\alpha')}(\alpha') \geq h_{\Omega(\alpha)}(\alpha)$ holds.  Hence $h(V,m,\Omega)=\sup\{ h_{\Omega(\alpha)}(\alpha,\Gamma): \alpha \in \Sigma_{\calM} \} = \sup\{ h_{\Omega(\alpha)}(\alpha,\Gamma): \alpha \text{ is a partition of } \prX\}$. It follows that it is enough to show that for all partitions $\alpha$ of $\prX$ and all invariant measure subalgebras $W$ of $\calM$, containing $\alpha$ we have $h^*(\alpha)=h_{W}(\alpha)$, since, clearly, $h^*(\alpha) = h_{\mu}(\alpha)$ for partitions. For any partition $\alpha$ and any invariant measure subalgebra $W$, containing $\alpha$ it follows by the pigeonhole principle that $\hat h_{W}(\alpha) = h^*(\alpha)$. For an arbitrary cover $\beta \in \Sigma_{W}$ there is a \emph{non-growing disjoint refinement} $\beta' \in \Sigma_{W}$, i.e. a partition $\beta'$ such that $\beta' \succeq \beta$ and $N(\beta') \leq N(\beta)$. Then, clearly, $\hat h_{W}(\beta') \leq \hat h_{W}(\beta)$.

It is obvious that for any partition $\alpha$ and any invariant measure subalgebra $W$, containing $\alpha$
\begin{equation*}
h_{W}(\alpha) = \inf\{ \sum\limits_{j=1}^n \hat h_{W}(\beta_j): \bigvee\limits_{j=1}^n \beta_j \succeq \alpha, n \in \N, \forall j \ \beta_j \in \Sigma_{W} \} \leq h^*(\alpha).
\end{equation*}
Now, pick any sequence $(\beta_j)_{j=1}^n$ of covers in $\Sigma_W$ s.t. $\bigvee\limits_{j=1}^n \beta_j \succeq \alpha$, and consider the respective non-growing disjoint refinements $(\beta_j')_{j=1}^n$. Then $\sum\limits_{j=1}^n \hat h_W(\beta_j') \leq \sum\limits_{j=1}^n \hat h_W(\beta_j)$ and it follows that it suffices to take the infinum over the sequences $(\beta_j)_{j=1}^n$ of partitions of $\prX$. It only remains to observe that for such $\alpha, (\beta_j)_{j=1}^n$
\begin{equation*}
h^*(\alpha) \leq h^*( \bigvee\limits_{j=1}^n \beta_j) \leq \sum\limits_{j=1}^n h^*(\beta_j)
\end{equation*}
by the standard monotonicity and subadditivity properties of $h_{\mu}$ on partitions.
\end{proof}

Now we prove a similar statement for topological dynamical systems and the measurement functor $\uM_{\TOP}$.
\begin{prop}
\label{p.comptopent}
Let $\bfX = (\uX, \calU)$  be in $\TOP$, and let $(\bfX,\pi) \in \ob(\Rep{\Gamma}{\TOP})$ be topological dynamical system,
equipped with the action $\pi$ of an amenable group $\Gamma$. Denote by $h_{\TOP}(\bfX,\pi)$ the topological entropy of the system $(\bfX,\pi)$. Then
\begin{equation*}
h_{\TOP}(\bfX,\pi) = h((\bfX,\pi), \uM_{\TOP}).
\end{equation*}
\end{prop}
\begin{proof}
Recall that the topological entropy of a finite open cover $\alpha$ of $\uX$ is defined as
\begin{equation}
\label{eq.entropcov}
h_{\TOP}(\alpha,\pi) = \lim\limits_{n \to \infty} \frac 1 {|F_n|} H(\alpha^{F_n}),
\end{equation}
where $H(\alpha)=\log \min\{ |\beta|: \beta \subseteq \alpha \text{ is a subcover}\}$ and $( F_n )_n$ is a F{\o}lner sequence. Then the topological entropy of the system $(\bfX,\pi)$ is
\begin{equation}
\label{eq.topentsys}
h_{\TOP}(\bfX,\pi) = \sup\{ h_{\TOP}(\alpha,\pi): \alpha \text{ finite open cover of } \uK \}.
\end{equation}

Let $(\overline \bfX,\overline \pi) = \uM_{\TOP}(\bfX,\pi)$ be the abstract dynamical lattice associated to system $(\bfX,\pi)$, where $(V,m,\Omega)$ is the measured lattice with localization.
Distributive lattice $V$ is the lattice of open subsets of $\bfX$, $m$ is equal to $1$ everywhere
except for the empty set, $m(\varnothing) = 0$, and for an open cover $\alpha$ we have defined
$\Omega(\alpha)$ as the smallest $\Gamma$-invariant topology containing family of open sets $\alpha$.

Let $\alpha$ be a finite open cover, and let $W$ be any invariant topology, containing $\alpha$. It is clear that $h^*(\alpha) = \log N(\alpha)$ and $\hat h_{W}(\alpha) = \log N(\alpha)$ as well. For a minimal subcover $\alpha'$ of $\alpha$ one concludes $\alpha' \succeq \alpha$ and $\alpha' \subseteq W$, so
\begin{align*}
h_{W}(\alpha) &= \inf\{ \sum\limits_{i=1}^n \hat h_{W} (\beta_i): \bigvee\limits_{i=1}^n \beta_i \succeq \alpha, n \in \N, \forall i \ \beta_i \in
\Sigma_{W} \} \leq \\
&\leq\hat h_{W}(\alpha') = H(\alpha)
\end{align*}

Now, pick an arbitrary sequence $(\beta_i)_{i=1}^n$ of covers in $\Sigma_W$ s.t. $\bigvee\limits_{i=1}^n \beta_i \succeq \alpha$. Then $$H( \alpha ) \leq H(\bigvee\limits_{i=1}^n \beta_i) \leq \sum\limits_{i=1}^n H(\beta_i) \leq \sum\limits_{i=1}^n \hat h_{W} (\beta_i)$$ by the standard monotonicity and subadditivity properties of $H$. It follows that $H(\alpha) = h_{W}(\alpha)$ and the proof of the statement is complete.
\end{proof}

\section{Conclusion and further questions}

\subsection{A structural counterxample}
\label{ss.structcex}
We begin by providing a counterexample that was promised in the introduction. Its designation was to show to that if one takes Palm's original notion of a measured distributive lattice entropy \emph{without localization}, then, in general, entropy of a sublattice can be bigger than the entropy of a whole lattice. We take $\Gamma$ to be trivial, this is not essential for our purposes. So consider measured distributive lattice $\bfV = (V,m)$, where $V$ is the distributive lattice of all subsets of $\{ a, b, c\}$, and $\mu$ is a probability measure taking value \emph{very} close to $1$ on atom $b$. Then by the properties of Palm's entropy $\tilde h$ we have $\tilde h(\bfV) \approx 0$, since the system $\bfV$ is \emph{very} close to a singleton as a system in the category $\PSP$.

Now, let $\bfW = (W,m)$ be a sublattice of $V$ containing subsets $$\varnothing, \{ a, b\}, \{ b \}, \{ b, c\}, \{ a, b, c\}$$ and carrying the induced measurement function. Then $\bfW$ is no longer an object of $\PSP$, it is closer to a topological system on $\{ a, b, c\}$ with $W$ being the lattice of opens. It is straightforward to see that such system has Palm's entropy $\tilde h(\bfW) \approx 1$, which is attained at the cover $\alpha:=\{\{ a, b\}, \{ b, c \} \}$.

We call this counterexample `structural' because of its nature: within a representation of a probabilistic system $\bfV$ as a measured lattice there exists a sublattice $\bfW$, which is not of a probabilistic origin. This has lead us to the idea of introducing the localized entropy $h_{\cdot}(\cdot)$. If one was to compute the localized entropy $h_W(\alpha)$ with $W,\alpha$ as above, one would also find that $h_W(\alpha) \approx 1$; but not if one takes $h_V(\alpha)$ instead. The reason is once again `structural': $V$ is the minimal spacial sublattice containing $\alpha$, while $W$ is not spacial at all.

We conjecture that one can also change the perspective a bit. In topological and measure-theoretic dynamics, given a cover $\alpha$, it \emph{does not matter} with respect to which spacial sublattice the entropy $h_{\cdot}(\alpha)$ is computed. So maybe the following approach is more natural: instead of a single spacial lattice $\Omega(\alpha)$, we could prescribe the whole `germ' $\Sub{\alpha}{\overline \bfA}$ of spacial lattices containing $\alpha$. Morphisms in this modified category of measured lattices with germs should also respect the germs. An even more informal interpretation could be that, given $\bfA \in \PSYS$ and $\uM(\bfA) \in \OBS$ for some $\uM \in [\PSYS,\OBS]_{\frakM}$, performing observations $\alpha$ in $\Sigma_{\overline \bfA}$ tells us something about the family $\Sub{\alpha}{\overline \bfA}$ of factors of $\bfA$, and not just about $\bfA$ itself, since we do not see $\bfA$ itself after all. Quantities like entropy/complexity should of course be stable in $\Sub{\alpha}{\overline \bfA}$ given the observations $\alpha$. However, it is not clear how one should enforce the stability of the relative entropy $h_{\cdot}(\alpha)$ on the germ of spacial sublattices containing $\alpha$.

\subsection{Representations on Banach lattices with quasi-interior points and the problem of observability} We have shifted the focus away from lattices of ideals of Banach lattices with quasi-interior points because the structure of these lattices is `obscure', and, furthermore, they do not play a significant role anymore once we arrive at the notion of category of measurement functors. However, we still wonder whether the correspondence suggested by Palm can be made into a measurement functor.

A closely related `philosophical' question is that of observability. It is often said (see \cite{nestruev}) that in reality our measurement tools are structures like $\mathrm{C}^{\infty}(\calM)$, $\Ce(\uK)$ or $\Ell{1}(\prX)$ on underlying phase spaces. We call distinguished regions of phase spaces \emph{observations}, and we can pass to observations by taking ideals of Banach lattices of observables in the particular cases of topological and measure-theoretic dynamical systems. Consequently, we can then define entropy. The question is how good one can combine these two ideas in general.

\bibliographystyle{acm}
\bibliography{measfunctbib}

\end{document}